\documentclass[12pt]{amsart}
\usepackage{graphicx, amssymb, epstopdf, amsmath, amssymb}
\DeclareMathAlphabet{\mathcal}{OMS}{cmsy}{m}{n}

\usepackage{mathptmx} % rm & math
\usepackage[scaled=0.90]{helvet} % ss
\usepackage{courier} % tt
\normalfont
\usepackage[T1]{fontenc}

\theoremstyle{plain}
\newtheorem{Thm}{Theorem}
\newtheorem*{Thms}{Theorems}
\newtheorem*{thm}{Theorem}
\newtheorem{Lem}[Thm]{Lemma}
\newtheorem{Cor}[Thm]{Corollary}

\theoremstyle{definition}
\newtheorem*{Def}{Definition}

\theoremstyle{remark}

\def\F{\mathbb{F}}

\def\cpb{\hbox{$\overline{{\mathbb{CP}}^2}$}}
\def\cpo{\hbox{${\mathbb{CP}}^1$}}

\def\cK{{\mathcal K}}

\def\BHM{Baykur, Hayano, and Monden }

\title{Unchaining Surgery, Branched Covers, and Pencils on Elliptic Surfaces}

\author{Terry Fuller}
\address{Department of Mathematics,
California State University, Northridge, 
Northridge, CA 91330 } 
\email{terry.fuller@csun.edu}

\begin{document}
\begin{abstract}
In \cite{Unchaining}, R. \.{I}nan\c{c} Baykur, Kenta Hayano, and Naoyuki Monden use a technique called {\em unchaining} to construct a family of simply connected symplectic 4-manifolds $X_g'(i)$, for all $g\geq 3$ and $0 \leq i \leq g-1$; among this family, the manifolds $X_g'(g-2)$ are shown to be symplectic Cababi-Yau 4-manifolds. They also show that each $X_g'(i) \# \cpb$ admits a pair of inequivalent genus~$g$ Lefschetz pencils. We show how to describe every $X_g'(i)$  as a 2-fold branched cover of a rational surface, and use this to prove that each $X_g'(i)$ is diffeomorphic to the elliptic surface $E(g-i)$. This has several notable consequences: (1) each symplectic Calabi-Yau they construct is diffeomorphic to K3; (2) for each $n \geq 3$ and $g\geq n$, the elliptic surface $E(n)$ admits a genus~$g$ Lefschetz pencil; and (3) for each $n \geq 3$ and $g\geq n$, the once blown up elliptic surface $E(n) \# \cpb$ admits a pair of inequivalent genus~$g$ Lefschetz pencils.
\end{abstract}
\maketitle
 
\section{Introduction}
Since the foundational work of Donaldson (\cite{Donaldson}) and Gompf (\cite{GompfStipsicz}) in the 1990s, Lefschetz pencils and fibrations are known to characterize symplectic 4-manifolds. In \cite{Unchaining}, R. \.{I}nan\c{c} Baykur, Kenta Hayano, and Naoyuki Monden construct a doubly indexed family of symplectic 4-manifolds $X_g'(i)$,  for all $g \geq 3$ and $0\leq i \leq g-1$. Their examples are constructed as the total spaces of symplectic genus $g$ Lefschetz pencils, through explicit factorizations of their monodromy. We review the specific factorizations which define $X_g'(i)$ below, but in the meantime summarize results from \cite{Unchaining} about these manifolds.

\begin{Thms}[\cite{Unchaining}] \label{UnchainingSummary}
For each $g\geq 3$ and $0 \leq i \leq g-1$, there is a genus $g$ Lefschetz pencil on $X_g'(i)$ with the following properties.
\begin{itemize}
\item[(a)] (\cite{Unchaining}, Lemma 4.7) The manifolds $X_g'(i)$ are simply connected, with Euler characteristic $e(X_g'(i))= 12(g-i)$, and signature $\sigma(X_g'(i))= -8(g-i)$.
\item[(b)] (\cite{Unchaining}, Lemma 5.6) The manifolds $X_g'(i)$ are spin if and only if $g-i$ is even.
\item[(c)]   (\cite{Unchaining}, Theorem 4.8) The manifolds $X_g'(g-1)$ are diffeomorphic to the rational elliptic surface $E(1)$.
\end{itemize}
\end{Thms}

These statements suggest the result which is the main theorem of this paper.

\begin{Thm} \label{main}
The manifolds $X_g'(i)$ are diffeomorphic to the elliptic surface $E(g-i)$.
\end{Thm}

This has some immediate Corollaries. In \cite{Unchaining}, \BHM note that when $g-i$ is even, $X_g'(i)$ is irreducible (since it is spin), but the irreducibility of $X_g'(i)$ for odd $g-i$ is left open. We now have

\begin{Cor}
$X_g'(i)$ is irreducible for all $g \geq 3, \ 0 \leq i \leq g-2.$
\end{Cor}

Additionally, in \cite{Unchaining}, the Kodaira dimensions of $X_g'(i)$ are computed only for the special cases of $g-3 \leq i \leq g-1$ and $g-i$ even. Our main theorem fills in the missing cases.

\begin{Cor} \label{Cor_Kodaira}
The symplectic Kodaira dimension of $X_g'(i)$ is
\[
	\kappa(X'_g(i)) = \begin{cases}
	-\infty &  \ \textit{if}  \ \  i = g - 1 \\
	\ \ 0 & \  \textit{if} \ \ i= g-2  \\
	\ \ 1 &  \ \textit{if} \ \ 0\leq i \leq g-3.
	\end{cases}
	\] 
\end{Cor}

An additional corollary concerns {\em symplectic Calabi-Yau} 4-manifolds. A complex Calabi-Yau surface is one with a trivial canonical class, and one can likewise define a symplectic Calabi-Yau 4-manifold to be one with a trivial symplectic canonical class. All known examples of symplectic Calabi-Yau manifolds are complex K3 surfaces or torus bundles over tori. Since any symplectic Calabi-Yau must have the rational homology type of these complex surfaces (\cite{Bauer}, \cite{Li_IMRN}), it is an intriguing open question if there exist {\em any} symplectic Calabi-Yau 4-manifolds which are not diffeomorphic to one of these (\cite{FriedlVidussi}, \cite{Li_JDG}). Baykur, Hayano, and Monden show that the manifolds $X_g'(g-2)$ are symplectic Calabi-Yau (\cite{Unchaining}, Corollary 4.10), and ask if they are diffeomorphic to the standard K3 surface. We have
\begin{Cor} \label{Cor_SCY}
The symplectic Calabi-Yau manifolds $X_g'(g-2)$ are diffeomorphic to K3.
\end{Cor}

In addition to its relevance to finding examples of symplectic Calabi-Yau manifolds, this result serves to illustrate the diversity of Lefschetz pencils on fixed 4-manifolds. The K3 surface is known to admit pencils of every genus \cite{Smith}, and it is noted in \cite{Unchaining} that the diffeomorphism $X_g'(g-1) \cong E(1)$ implies that the same is true for the rational elliptic surface $E(1)$. The author is not aware of any other such examples. We now have

\begin{Cor} \label{all_g}
For all $n \geq 3$, the elliptic surface $E(n)$ admits a genus~$g$ Lefschetz pencil for every $g \geq n$.
\end{Cor}

A deeper related application concerns finding inequivalent Lefschetz pencils on a given 4-manifold with the same topological data (i.e. genus and number of base points). By using the braiding lantern substitution technique of \cite{Multisections}, Baykur, Hayano, and Monden prove
\begin{thm}[\cite{Unchaining}, Corollary 6.4] For all $g \geq 3$ and $0 \leq i \leq g-1$, the manifold $X_g'(i) \# \cpb$ admits a pair of inequivalent genus~g Lefschetz pencils. In particular, the manifold $E(1)  \# \cpb$ admits a pair of inequivalent genus~$g$ Lefschetz pencils for all $g \geq 3$.
\end{thm}

Theorem~\ref{main} strengthens this result.

\begin{Cor} \label{all_g_inequiv}
For all $n \geq 3$, the once blown up elliptic surface $E(n) \# \cpb$ admits a pair of inequivalent genus~$g$ Lefschetz pencils for all $g \geq n$.
\end{Cor}
\noindent Of course, the conclusions of Corollaries \ref{all_g} and \ref{all_g_inequiv} apply to blow ups of $E(n)$ and $E(n) \# \cpb$ at base points, as well. 

The method of proof of Theorem~\ref{main} exploits the natural 2-fold symmetry of Baykur, Hayano, and Monden's construction. We begin by blowing up the pencil on $X_g'(i)$ to obtain an associated Lefschetz fibration $X_g(i)$, and use this symmetry to represent $X_g(i)$ as a 2-fold branched cover of a rational surface. A sequence of handle slides in the base of this cover allows one to find and blow down the required number of exceptional sections, and we arrive at a branched cover description of $X_g'(i)$. The branch surface of this cover is represented as a banded unlink diagram, of the sort studied in \cite{BandedUnlink}, with an explicitly drawn ribbon surface as (most of) the branch locus. We then use various band moves to obtain an isotopy of the branch surface, yielding a branched cover description that is recognized as one for elliptic surfaces.

In Section~\ref{sec:BandedUnlink} we discuss banded unlink diagram descriptions of embedded surfaces, and review the moves on these diagrams that we will employ in the proof. The following section reviews the topology of Lefschetz pencils and fibrations.  Finally, in Section~\ref{sec:proof}, we define the manifolds $X_g'(i)$ and $X_g(i)$, and give the proof of Theorem~\ref{main}.

%%%%%%%%%%%%%%%%%%%%%%%%%%%%%
\section{Banded Unlink Diagrams} \label{sec:BandedUnlink}
%%%%%%%%%%%%%%%%%%%%%%%%%%%%%

In this section we review the notion of a banded unlink diagram \cite{BandedUnlink}.  This describes a closed surface embedded in a closed 4-manifold $X$. Banded unlink diagrams can be defined using any handlebody description for $X$, but since in our application $X$ will lack 1- and 3-handles, we only discuss that setting here. 

Suppose $X$ is obtained by attaching $n$ 2-handles to a single 0-handle, and then attaching one 4-handle. The manifold $X$ can be depicted by a Kirby diagram $\cK$ consisting of an $n$-component framed link in $S^3$. Let $X_0$ denote the boundary of the 0-handle, and $X_1$ the union of the 0- and 2-handles. Of course, both $\partial(X_0)$ and $\partial(X_1)$ are $S^3$, and $\partial(X_1)$ can be described as the result of a surgery of $S^3$ along the components of $\cK$.

Let $L$ be a link in the exterior $E(\cK)$. Since $L$ avoids the attaching region of the 2-handles, we can view $L$ as a link in $\partial(X_0)$  and in $\partial(X_1)$. In a banded unlink diagram, we begin with an unlink in $E(\cK)$, and form a ribbon surface by attaching a disjoint collection of bands to the spanning disks of the unlink; $L$ is the link that results from the band surgery to the unlink, and we may push the interior of the ribbon surface into $X_0$ to get an embedded surface. In a banded unlink diagram, we also require that $L$ bounds a collection of disjoint disks in $\partial(X_1)$. In this way, the ribbon surface that $L$ bounds can be capped off by these disks, giving a closed surface in $X$.

In \cite{BandedUnlink}, Mark Hughes, Seungwon Kim, and Maggie Miller give a complete set of moves for banded unlink diagrams of isotopic surfaces in a 4-manifold. As we will apply these to manifolds without 1-handles, we review only the moves that we use later: band slides, band swims, 2-handle band swims, and 2-handle band slides. These are shown in Figure~\ref{buds}.
\begin{figure}[ht]
\begin{center}
\includegraphics[width=6in]{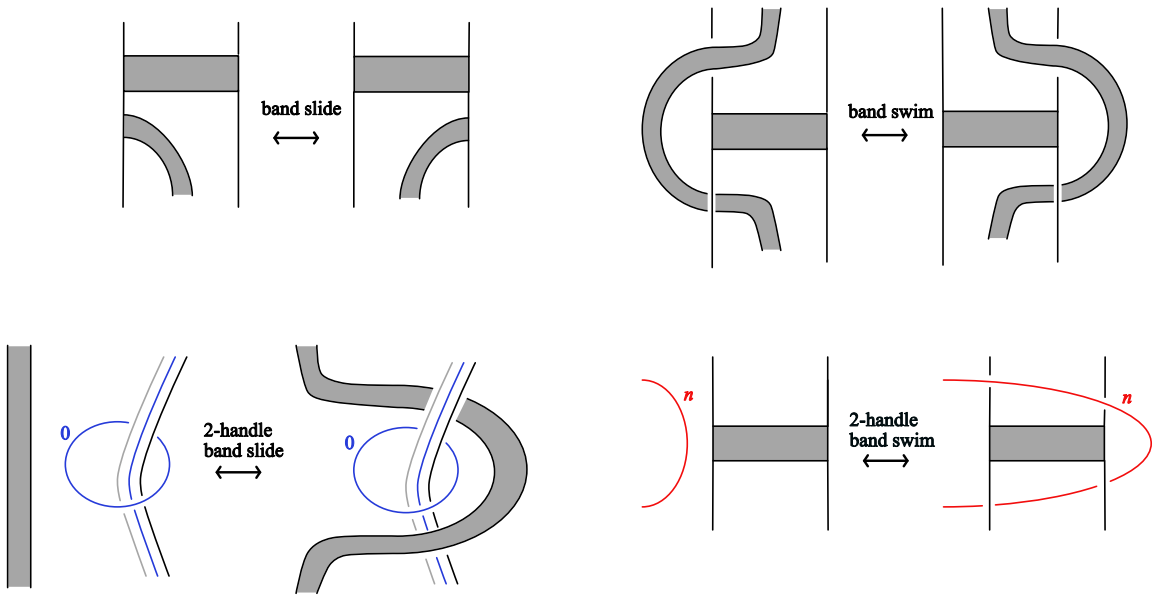}
\caption{In the two swim moves, the band/attaching circle passes lengthwise through the interior of the horizontally drawn band.}
\label{buds}
\end{center}
\end{figure}
\noindent (The 2-handle band slides in Figure~\ref{buds} can be done with any knotted attaching circle and any framing, following the usual rules of Kirby Calculus; the 0-framed unknot pictured here is all that will be used later. The strands running through the attaching circle of the handle can represent other handles, bands, or unlink components.)

Two particular iterations of these moves will be used often, and are shown below. For reference, we refer to these as {\em band dives} and {\em 2-handle band dives}.
\begin{figure}[ht]
\begin{center}
\includegraphics[width=5in]{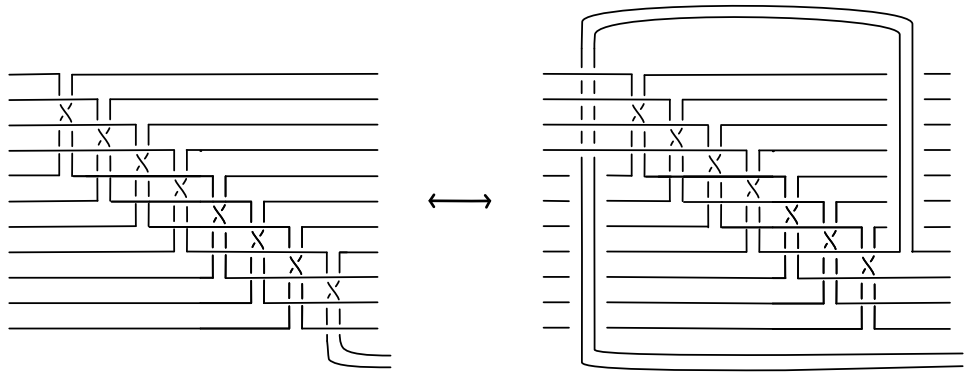}
\caption{A band dive.}
\label{band dive}
\end{center}
\end{figure}
\begin{figure}[ht]
\begin{center}
\includegraphics[width=6in]{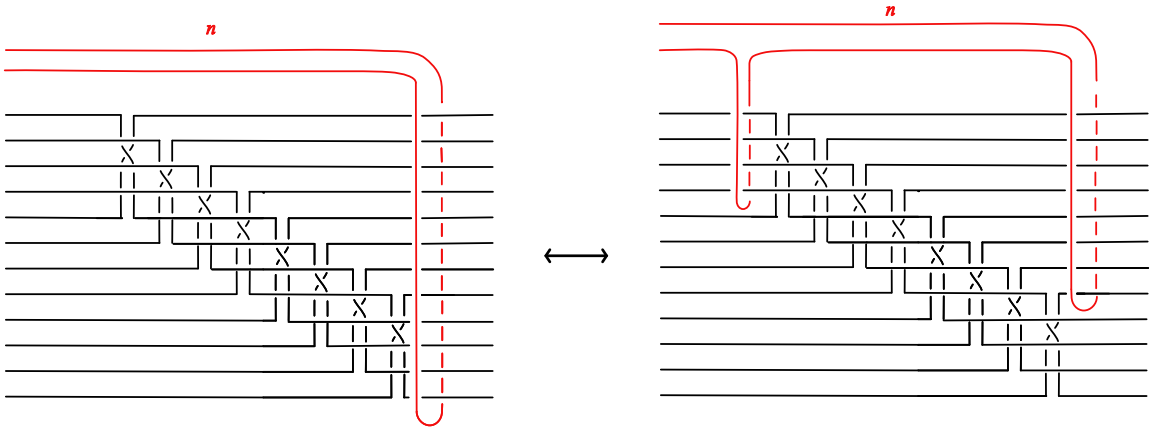}
\caption{A 2-handle band dive.}
\label{2h band dive}
\end{center}
\end{figure}
These will later be used left-to-right, as drawn in Figures~\ref{band dive} and \ref{2h band dive}, although their name -- and easiest visualization, perhaps --  comes from regarding them right-to-left, as a sequence of swims descending through the half twisted bands.

%%%%%%%%%%%%%%%%%%%%%%%%%%%%%%%
\section{Lefschetz Pencils and Fibrations} \label{sec:Lefschetz}
%%%%%%%%%%%%%%%%%%%%%%%%%%%%%%%

In this section we review the definitions of Lefschetz fibrations and Lefschetz pencils, and discuss the topology of these structures. A more comprehensive description of the topology of Lefschetz fibrations and pencils can be found in \cite{GompfStipsicz}.

We denote a closed oriented genus $g$ surface as $\Sigma_g$, and a compact oriented genus $g$ surface with $n$ boundary components as $\Sigma_g^n$. Their mapping class groups will be denoted as $\Gamma_g$ and $\Gamma_g^n$, respectively. We will also denote a sphere with $m$ marked points by $\Sigma_{0,m}$, and its mapping class group by $\Gamma_{0,m}$. 
\begin{Def}
Let $W$ be a compact oriented smooth 4-manifold, and $C$ a compact oriented smooth surface. A proper smooth map $f:W \to C$ is a {\em Lefschetz fibration} if 
\begin{itemize}
\item[(1)] the critical points of $f$ lie in the interior of $W$; and
\item[(2)] for each critical point of $f$ in $W$, there are complex coordinate charts agreeing with the orientations on $W$ and $C$ such that locally $f$ can be expressed as $f(z_1,z_2)=z_1^2+z_2^2$.
\end{itemize}
\end{Def}

In this paper, we will only encounter $C=S^2$ or $D^2$. 

\begin{Def} Let $W'$ be a closed oriented smooth 4-manifold
Let $B \subset W'$ be a finite set of points. A smooth map $f:W' \setminus B \to \cpo$ is a {\em Lefschetz pencil} if
\begin{itemize}
\item[(1)] for each critical point of $f$ in $W' \setminus B$, there are complex coordinate charts agreeing with the orientations on $W'$ and $\cpo$ such that locally $f$ can be expresses as $f(z_1,z_2)=z_1^2+z_2^2$; and
\item[(2)] for each point of $B$, there is a complex coordinate chart on $W'$ and an identification of the base as $\cpo$ such that locally $f$ can be expressed as $f(z_1,z_2)=[z_1:z_2]$.
\end{itemize}
The existence of a Lefschetz pencil $f:W' \setminus B \to \cpo$ will be described by saying that there is a Lefschetz pencil {\em on $W'$}.
\end{Def} 

The points of $B$ are called {\em base points} of the Lefschetz pencil. A Lefschetz pencil with $B=\emptyset$ is a Lefschetz fibration over $\cpo\cong S^2$. If $B\neq \emptyset$, we can blow up $W'$ at each basepoint to get $W$, and the Lefschetz pencil on $W'$ becomes a Lefschetz fibration $W \to \cpo\cong S^2$. 

It is a consequence of these definitions that for a Lefschetz fibration, a regular fiber $f^{-1}(x)$ is a closed genus $g$ surface. For a Lefschetz pencil with $n>0$ base points, $f^{-1}(x)$ is not compact, and we instead consider  $f^{-1}(x)\cap(W'\setminus(U_1 \cup \cdots \cup U_n))$, where $U_i$ is an open ball about the base point in each coordinate chart with property (2) above. This fiber will be a compact genus $g$ surface with $n$ boundary components. We refer to {\em genus $g$} Lefschetz fibrations or pencils, accordingly.

Lefschetz pencils and fibrations are understood topologically through monodromy factorizations. Let $x_1, \ldots, x_\mu$ be the critical values for $f$. We assume, without loss of generality, that each critical point of $f$ lies in a separate fiber. For a pencil, we select a regular value $x_0 \in \cpo$, and a disjoint collection of arcs $\gamma_i$ from $x_0$ to $x_i$, for each $i=1,\ldots, \mu$. (We also assume each $\gamma_i$ avoids the other critical points.) We further assume the arcs $\gamma_1, \ldots, \gamma_\mu$ appear in this order as we travel in a small circle about $x_0$. For each $i$, we consider a loop that begins at $x_0$, travels along $\gamma_i$, then counterclockwise around a small circle centered at $x_i$, and back to $x_0$ along $\gamma_i$. Using an identification of $f^{-1}(x_0)$ with $\Sigma_g^n$, the monodromy of $f$ along this loop is known to be a right-handed Dehn twist $t_{c_i}$ along a simple closed curve $c_i \subset \Sigma_g^n$ (\cite{Kas}). The curve $c_i$ is called a {\em vanishing cycle}. To get a global description of a Lefschetz pencil, these local models must fit together according to the equation $t_{c_1} \ldots t_{c_\mu} = t_{\delta_1} \ldots t_{\delta_n}$ in $\Gamma_g^n$, where $\delta_j$ denotes a right-handed  Dehn twist about a curve parallel to the $j$th boundary component of $\Sigma_g^n$. Conversely, given any factorization in $\Gamma_g^n$ of $t_{\delta_1} \ldots t_{\delta_n}$ as a product of right-handed Dehn twists, one can construct a Lefschetz pencil with monodromy prescribed by the factorization. 

When working with Lefschetz fibrations, one has a similar description of the local monodromy about a critical value $x_i$ as a right-handed Dehn twist $t_{c_i}$ about a simple closed curve $c_i \subset \Sigma_g$. To form a global Lefschetz fibration over $S^2$, the local monodromies must concatenate to form a relation $t_{c_1} \ldots t_{c_\mu} =1$ in $\Gamma_g$. 

Any particular monodromy description of a Lefschetz pencil is far from unique, as it depends on a choice of identification of a regular fiber, as well as on a system of arcs $\gamma_i$. Modifying these choices translates into a simple set of moves on factorizations in $\Gamma_g^n$ (see \cite{GompfStipsicz}), and two factorizations related in this way are said to be {\em Hurwitz equivalent}. 

There is a straightforward relationship between a monodromy factorization of a Lefschetz pencil on $W'$ and that of the Lefschetz fibration $W \to S^2$ obtained by blowing up $W'$ at all base points. Under the homomorphism $\Gamma_g^n \to \Gamma_g$ obtained by capping off each boundary component of $\Sigma_g^n$ with a disk, Dehn twists about the boundary parallel curves $\delta_j$ become trivial in $\Gamma_g$. A monodromy factorization $t_{c_1} \ldots t_{c_\mu} = t_{\delta_1} \ldots t_{\delta_n}$ in $\Gamma_g^n$ for the pencil on $W'$ then gives a monodromy factorization $t_{c_1} \ldots t_{c_\mu} =1$ in $\Gamma_g$ for the fibration $W \to S^2$. 

A monodromy factorization of a genus $g$ Lefschetz fibration $f:W \to S^2$ also leads to a handlebody description of $W$, in a well-understood way (\cite{GompfStipsicz}). One begins with a handlebody description of $\Sigma_g \times D^2$ consisting of a 0-handle, 1-handles, and 2-handles. Given a factorization $t_{c_1} \ldots t_{c_\mu} =1$ in $\Gamma_g$, we form $\Sigma_g \times D^2 \cup (\cup_{i=1}^\mu H_i^2),$ where each $H_i^2$ is a 2-handle attached along the vanishing cycle $c_i$ in a separate fiber $\Sigma_g \times \{ {\mathrm{point}} \} \subset \Sigma_g \times S^1 = \Sigma_g \times \partial D^2$. The 2-handles are attached along the $S^1$ factor in the order they appear in the factorization, and they have framing $-1$ relative to the framing on $c_i$ induced by the product $\Sigma_g \times S^1$. Following these handle attachments, we have a handlebody describing a Lefschetz fibration over $D^2$ with the prescribed monodromy factorization. The boundary of $\Sigma_g \times D^2 \cup (\cup_{i=1}^\mu H_i^2)$ is $\Sigma_g$-bundle over $S^1$ with monodromy $t_{c_1} \ldots t_{c_\mu}$; because this is isotopic to the identity, this boundary is diffeomorphic to $\Sigma_g \times S^1$. Hence we can extend the Lefschetz fibration to be one over $S^2$ by attaching the trivial fibration $\Sigma_g \times D^2 \to D^2$ along $\Sigma_g \times S^1$. This final attachment adds one or more 2-handles, 3-handles, and a 4-handle. 

A technique for constructing new Lefschetz pencils or fibrations from old is {\em monodromy substitution}. Given a monodromy factorization, a monodromy substitution swaps a subword of the factorization with a different (but equal, in $\Gamma_g^n$ or $\Gamma_g$) product of right-handed Dehn twists. In \cite{Unchaining}, Baykur, Hayano, and Monden employ this operation using the {\em odd chain relation}: suppose $c_1, c_2, \ldots, c_{2h+1}$ are a collection of simple closed curves on $\Sigma_g^n$ or $\Sigma_g$ that form a chain; that is, $c_i$ and $c_{i+1}$ intersect in one point for all $i$, and $c_i$ and $c_j$ are disjoint otherwise.  A regular neighborhood of $c_1 \cup \cdots \cup c_{2h+1}$ is a subsurface $S$ homeomorphic to $\Sigma_h^2$. The chain relation is $(t_{c_1} t_{c_2} \cdots t_{c_{2h+1}})^{2h+2} = t_{b_1} t_{b_2}$, where $b_1$ and $b_2$ are the boundary components of $S$. Using this relation to replace a  subword in a monodromy factorization given by the left hand side of the chain relation with the two Dehn twists on the right is referred to as {\em unchaining}.

{\bf Realizing Hyperelliptic Lefschetz fibrations as branched covers.}
Let $\iota: \Sigma_g \to \Sigma_g$ be the hyperelliptic involution, and $\pi: \Sigma_g \to \Sigma_{0,2g+2}$ the branched covering that is the quotient of $\iota$. A Lefschetz fibration on $W\to S^2$ is {\em hyperelliptic} if it is Hurwitz equivalent to one with a mondromy factorization where each vanishing cycle $c_i$ satisfies $\iota(c_i)=c_i$. If all $c_i$ are nonseparating, then $W$ is a 2-fold branched cover of an $S^2$-bundle over $S^2$, with the Lefschetz fibration map obtained as the composition of this cover with the bundle projection (\cite{Hyperelliptic}). This cover is crucial to the proof of Theorem~\ref{main}, and we review the details.

Since all $c_i$ are non-separating and symmetric, the factorization $t_{c_1} \ldots t_{c_\mu} = 1$ is the lift of the relation $h_{\pi(c_1)} \ldots h_{\pi(c_\mu)} =1$ in $\Gamma_{0,2g+2}$, where $h_{\pi(c_i)}$ is a right-handed disk twist about the arc $\pi(c_i)$ in $\Sigma_{0,2g+2}$. The factorization $h_{\pi(c_1)} \ldots h_{\pi(c_\mu)}$ can be used to construct a ribbon surface in $S^2 \times D^2$, for which the cover branched over that surface is a Lefschetz fibration over $D^2$ with the required monodromy factorization. The Birman-Hilden Theorem (see \cite{BirmanHilden}, \cite{FarbMargalit}) then implies that we can always extend this cover by attaching a trivial covering of $\Sigma_g \times D^2$ over $S^2 \times D^2$, resulting in $W$ covering an $S^2$-bundle over $S^2$ branched over a closed surface. 

In practice, the base and branch set of this cover can be explicitly drawn as a banded unlink diagram. In $S^2 \times D^2$, represented as a Kirby diagram by a 0-framed unknot, we begin with $2g+2$ disks representing $\{{\mathit{point}}\} \times D^2$, drawn as meridians to the unknot, with their interiors pushed into the 0-handle. The branched cover of $S^2 \times D^2$ over these disks is $\Sigma_g \times D^2$, restricting to the hyperelliptic quotient in each fiber. A ribbon surface is then constructed by attaching left-handed half-twisted bands so that the core of each band is the arc $\pi(c_i)$ in $S^2 \times \{{\mathit{point}}\}$. By the method in \cite{AkbulutKirby}, in the 2-fold cover of $S^2 \times D^2$ branched over this ribbon surface, each added band lifts to a 2-handle attached along $c_i$, with relative framing $-1$.  Thus the lift of $S^2 \times D^2$ branched over the ribbon surface is the total space of a Lefschetz fibration over $D^2$, with monodromy factorization  $t_{c_1} \ldots t_{c_\mu}$. On the boundary, we have a $\Sigma_g$-bundle over $S^1$ covering an $S^2$-bundle over $S^1$, each with monodromy isotopic to the identity. To extend the branched covering over $W$, it is necessary to find a {\em fiber-isotopy} of the factorization $t_{c_1} \ldots t_{c_\mu}$ to the identity (i.e. an isotopy through homeomorphisms which are all fiber-preserving with respect to $\pi$): using a given fiber-isotopy to the identity, we can then identify the branched covering on the boundary as $\pi \times {\mathrm{id}}: \Sigma_g \times S^1 \to S^2 \times S^1$ and extend the covering as $\pi \times {\mathrm{id}}: \Sigma_g \times D^2 \to S^2 \times D^2$. The attachment of $S^2 \times D^2$ to the base matches the boundary of disks $\{{\mathit{point}}\} \times D^2$ to the boundary of the ribbon surface, and in this way we get a closed surface as branch set. The extension attaches a 2-handle union a 4-handle to the diagram of the base, with the 2-handle attached as a meridian to the 0-framed 2-handle. When working with examples, the braid factorization $h_{\pi(c_1)} \ldots h_{\pi(c_\mu)}$ plays a valuable role. The necessary fiber-isotopy to the identity can often be seen by simply observing that the braid factorization is isotopic to the identity by an isotopy that fixes the branch points at all times, in which case one obtains a fiber-isotopy of $t_{c_1} \ldots t_{c_\mu}$ to the identity as its lift. We can also use the braid factorization to compute the framing of the second attached 2-handle and to see how the attaching circle links the boundary of the branch surface. To do this, we select a reference point $\ast \in \Sigma_{0,2g+2} \setminus B^2_{2g+2}$, where $B^2_{2g+2}$ is a disk containing the branch points, and track a framed neighborhood of $\ast$ through the isotopy of $d_{h(c_1)} \ldots d_{h(c_\mu)}$ to the identity.

In \cite{Hyperelliptic}, it was shown how to modify this branched covering description of a hyperelliptic Lefschetz fibration to accomplish an unchaining monodromy substitution. Although the procedure in \cite{Hyperelliptic} was described only for even unchaining substitutions, the method applies equally well to the odd unchaining substitutions considered here.

%%%%%%%%%%%%%%%%%%%%%%%%%%%%
\section{The proof of Theorem \ref{main}} \label{sec:proof}
%%%%%%%%%%%%%%%%%%%%%%%%%%%%
We are now ready to describe the manifolds $X_g'(i)$ constructed by Baykur, Hayano, and Monden, and prove that they are diffeomorphic to the elliptic surfaces $E(g-i)$. 

\subsection{The Manifolds $X'_g(i)$ and $X_g(i)$.} In \cite{Unchaining}, Baykur, Hayano, and Monden construct their infinite family of Lefschetz pencils by explicit monodromy factorization. Their factorizations use Dehn twists about the curves on $\Sigma_g^{2(i+1)}$ shown below. We abbreviate the product of boundary curve twists as $\Delta =  t_{\delta_{i+1}} \cdots t_{\delta_2} t_{\delta_1} t_{\delta_{i+1}^\prime} \cdots t_{\delta_2^\prime} t_{\delta_1^\prime}$, and also let $D_g=t_{d_4} t_{d_5} \cdots t_{d_{2g+1}}$ and $E_g=t_{e_{2g+1}} \cdots t_{e_5} t_{e_4}.$
\begin{center}
\includegraphics{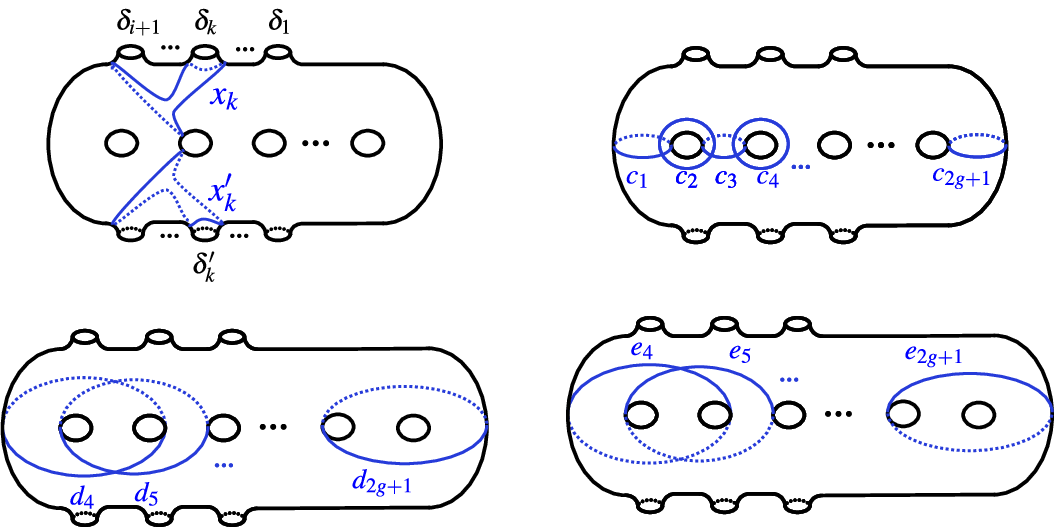}
\end{center}

\begin{thm}[\cite{Unchaining}, Theorem 4.6] \label{BHMpencils} 
For each $g \geq 3$ and $0\leq i \leq g-1$, there are symplectic genus $g$ Lefschetz pencils on $X'_g(i)$ with monodromy factorizations in $\Gamma_g^{2(i+1)}$ given by:

\begin{align}
\Delta=& \begin{cases}
D_g E_g t_{x_{i+1}} \cdots t_{x_2} t_{x_1} t_{x_{i+1}^\prime} \cdots t_{x_2^\prime} t_{x_1^\prime} (t_{c_1} t_{c_2} t_{c_3})^{4(g-i)} & (g \ \mathrm{odd}) \\
D_g E_g t_{x_{i+1}} \cdots t_{x_2} t_{x_1} t_{x_{i+1}^\prime} \cdots t_{x_2^\prime} t_{x_1^\prime} (t_{c_1} t_{c_2} t_{c_3})^{4(g-1-i)+2} (t_{c_3} t_{c_2} t_{c_1})^2 & (g \ \mathrm{even}). \label{BHMmonodromy_p}
\end{cases}
\end{align}
\end{thm}
\noindent (We have cyclically permuted the righthand side of (\ref{BHMmonodromy_p}) from its expression in \cite{Unchaining}.) 

If we cap off each boundary component of $\Sigma_g^{2(i+1)}$ with a disk, each of the curves $x_j$ (and $x_j'$) become parallel copies of a curve $x$ (and $x'$, respectively) on $\Sigma_g$. 
\begin{center}
\includegraphics{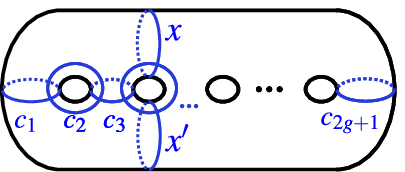}
%\caption{Curves on $\Sigma_g^{2(i+1)}.}
\end{center}
From (\ref{BHMmonodromy_p}) we see that the monodromy factorization of the Lefschetz fibration $X_g(i)\to S^2$ is 
\begin{align}
1=& \begin{cases}
D_g E_g (t_{x})^{i+1}  (t_{x^\prime})^{i+1} (t_{c_1} t_{c_2} t_{c_3})^{4(g-i)} & (g \ \mathrm{odd}) \\
D_g E_g (t_{x})^{i+1}  (t_{x^\prime})^{i+1} (t_{c_1} t_{c_2} t_{c_3})^{4(g-1-i)+2} (t_{c_3} t_{c_2} t_{c_1})^2 & (g \ \mathrm{even}). 
\label{BHMmonodromy_f}
\end{cases}
\end{align}

As it will play a role later, we review Baykur, Hayano, and Monden's derivation of this monodromy factorization. They begin with the full chain relation $$t_{\delta_1} t_{\delta_1'} =(t_{c_1} t_{c_2} \cdots t_{c_{2g+1}})^{2g+2}$$ in $\Gamma_{g}^2$. This is well known to be the monodromy of a pencil with two base points on a complex surface $Z'_g$ of general type. Through a series of Lemmas, they show this is Hurwitz equivalent to the factorization 
\begin{align}
t_{\delta_1} t_{\delta_1^\prime} =& \begin{cases}
 D_g E_g (t_{c_1} t_{c_2} t_{c_3})^{4g}  (t_{c_5} t_{c_6} \cdots t_{c_{2g+1}})^{2g-2} & (g \ \mathrm{odd}) \\
D_g E_g (t_{c_1} t_{c_2} t_{c_3})^{4(g-1)+2} (t_{c_3} t_{c_2} t_{c_1})^2  (t_{c_5} t_{c_6} \cdots t_{c_{2g+1}})^{2g-2}  & ( g \ \mathrm{even}).
\label{BHMchain_Heq_p}
\end{cases}
\end{align}
They then apply unchaining monodromy substitutions to this factorization, $i$ times to the subword $(t_{c_1} t_{c_2} t_{c_3})^{4}$, and once to $(t_{c_5} t_{c_6} \cdots t_{c_{2g+1}})^{2g-2}$. In addition, a clever inductive use of the lantern relation shows that this relation has a lift from $\Gamma_g^2$ to $\Gamma_g^{2(i+1)}$, providing enough sections of the pencil to allow for the computation of the symplectic Kodaira dimension for some of the resulting 4-manifolds, and giving the relation (\ref{BHMmonodromy_p}).

We give separate proofs that $X'_g(i) \cong E(g-i)$ for $g$ odd and even. Each proof will have two stages: representing $X_g'(i)$ as 2-fold branched cover, followed by modifications of the base that realize the diffeomorphism.

\subsection{The proof for odd $g$}
\subsubsection{Representing $X_g'(i)$ as a branched covering.}
Let $\F_n$ denote the $n$th Hirzebruch surface. We begin by discussing how to represent $X'_g(i)$ for odd $g$ as the 2-fold branched cover of the rational surface $\F_{i+1}$, branched over an embedded surface. The base of the covering and the branch surface will be represented as a banded unlink diagram.

Recalling the derivation of the monodromy factorization (\ref{BHMmonodromy_p}) above, we discuss first the Lefschetz fibration $Z_g \to S^2$ that comes from blowing up the Lefschetz pencil defined by (\ref{BHMchain_Heq_p}). This Lefschetz fibration on $Z_g$ has monodromy given by the relation

\begin{equation} \label{BHMchain_Heq_f_odd}
D_g E_g (t_{c_1} t_{c_2} t_{c_3})^{4i} (t_{c_5} t_{c_6} \cdots t_{c_{2g+1}})^{2g-2} (t_{c_1} t_{c_2} t_{c_3})^{4(g-i)} = 1
\end{equation}
 in $\Gamma_g$. This is a hyperelliptic Lefschetz fibration, and from the discussion in Section~\ref{sec:Lefschetz}, we see that $Z_g$ can be described as the 2-fold cover of $\F_1$ branched over the surface described in Figure~\ref{base_odd_chains}. The visible part of the branch surface is the ribbon surface consisting of $2g+2$ horizontal disks together with the collection of bands $C_4$, $C_{2g-2}$, and $D_{2g+2}$ defined in Figures~\ref{Cbraid} and \ref{Dbraid}. (The exponents for $C_4$ denote repeated copies.)
\begin{figure}[ht]
%\begin{center}
\includegraphics[width=4.8in]{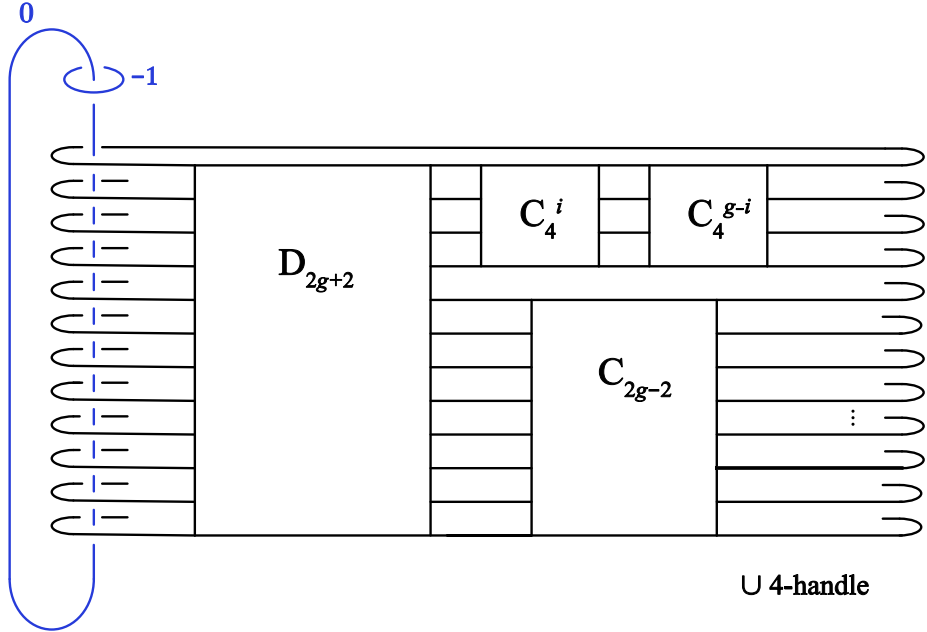}
\caption{The 2-fold branched cover is $Z_g$.}
\label{base_odd_chains}
%\end{center}
\end{figure}
\begin{figure}[ht]
%\begin{center}
\includegraphics[width=3.8in]{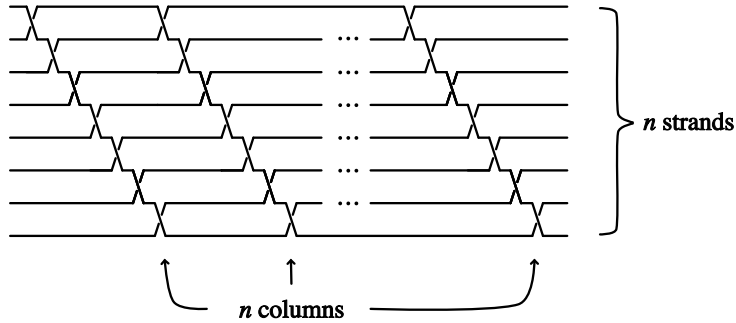}
\caption{The braid $C_n$}
\label{Cbraid}
%\end{center}
\end{figure}
\begin{figure}[ht]
%\begin{center}
\includegraphics[width=3in]{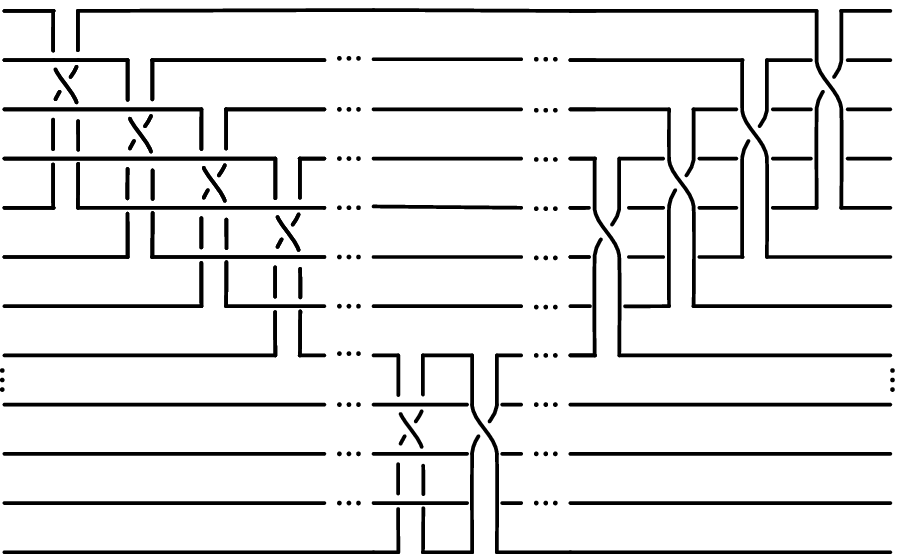}
\caption{The braid $D_n$, with $n$ strands.}
\label{Dbraid}
%\end{center}
\end{figure}
The branched cover of the 0-handle union the 0-framed 2-handle branched over the ribbon surface is a Lefschetz fibration over $D^2$ with monodoromy given by (\ref{BHMchain_Heq_f_odd}). It can be checked directly using the Alexander method (see \cite{FarbMargalit}) that the projection of (\ref{BHMchain_Heq_f_odd}) to a homeomorphism of $\Sigma_{0,2g+2}$ equals a right-handed Dehn twist about a circle which encloses all marked branch points. This is isotopic to the identity by an isotopy that fixes all branch points throughout, providing a fiber-isotopy to the identity, as required. This isotopy also fixes a reference point $\ast\in \Sigma_{0,2g+2} \setminus B_{2g+2}^2$, and rotates a framed neighborhood of $\ast$ once in a left-handed direction. Thus if we attach the second 2-handle as shown in Figure~\ref{base_odd_chains}, along a meridian with framing $-1$, we match $2g+2$ disks to the boundary of the ribbon surface, and we see $Z_g$ as the cover of the surface given as a banded unlink diagram, as claimed.

We now consider unchaining substitutions on (\ref{BHMchain_Heq_f_odd}), $i$ times  on the subword $(t_{c_1} t_{c_2} t_{c_3})^{4i}$ and once on $(t_{c_5} t_{c_6} \cdots t_{c_{2g+1}})^{2g-2}$. Doing so yields Baykur, Hayano, and Monden's relation (\ref{BHMmonodromy_f}) that defines the Lefschetz fibration $X_g(i)$.  As described in \cite{Hyperelliptic}, we can realize this substitution pictorially by ``blowing up'' the chain boxes in Figure~\ref{base_odd_chains}; that is, by replacing them with $-1$-framed 2-handles, as shown in Figure~\ref{base_odd_0}. 
\begin{figure}[ht]
\begin{center}
\includegraphics[width=4.8in]{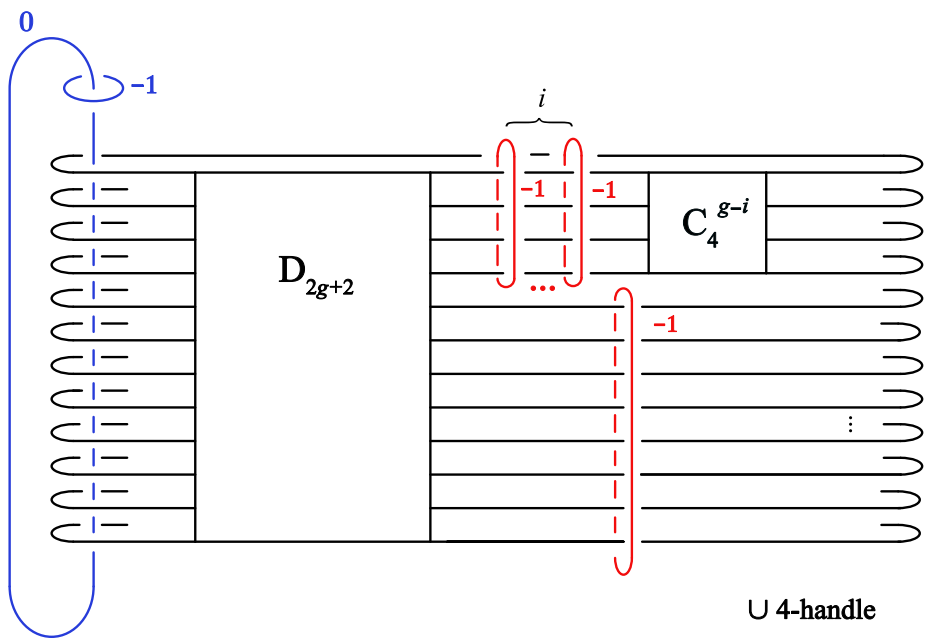}
\caption{The 2-fold branched cover is $X_g(i)$.}
\label{base_odd_0}
\end{center}
\end{figure}
Each of the newly introduced \mbox{2-handles} will lift to two 2-handles with relative product framing $-1$, attached along the pair of vanishing cycles $x$ and $x'$. This Figure still represents a banded unlink diagram, with $2g+2$ disks in the 4-handle, attached to the boundary of the ribbon surface. Thus $X_g(i)$ is the 2-fold cover of $\F_1 \# (i+1) \cpb$ branched over the surface shown in Figure~\ref{base_odd_0}.

We next execute a series of moves to the base of the branched covering. We begin by isotoping the newly added 2-handles by swinging them around the back of the ribbon surface so that they appear on the left, as in Figure~\ref{base_odd_1}. 
\begin{figure}[ht]
\begin{center}
\includegraphics[width=4.8in]{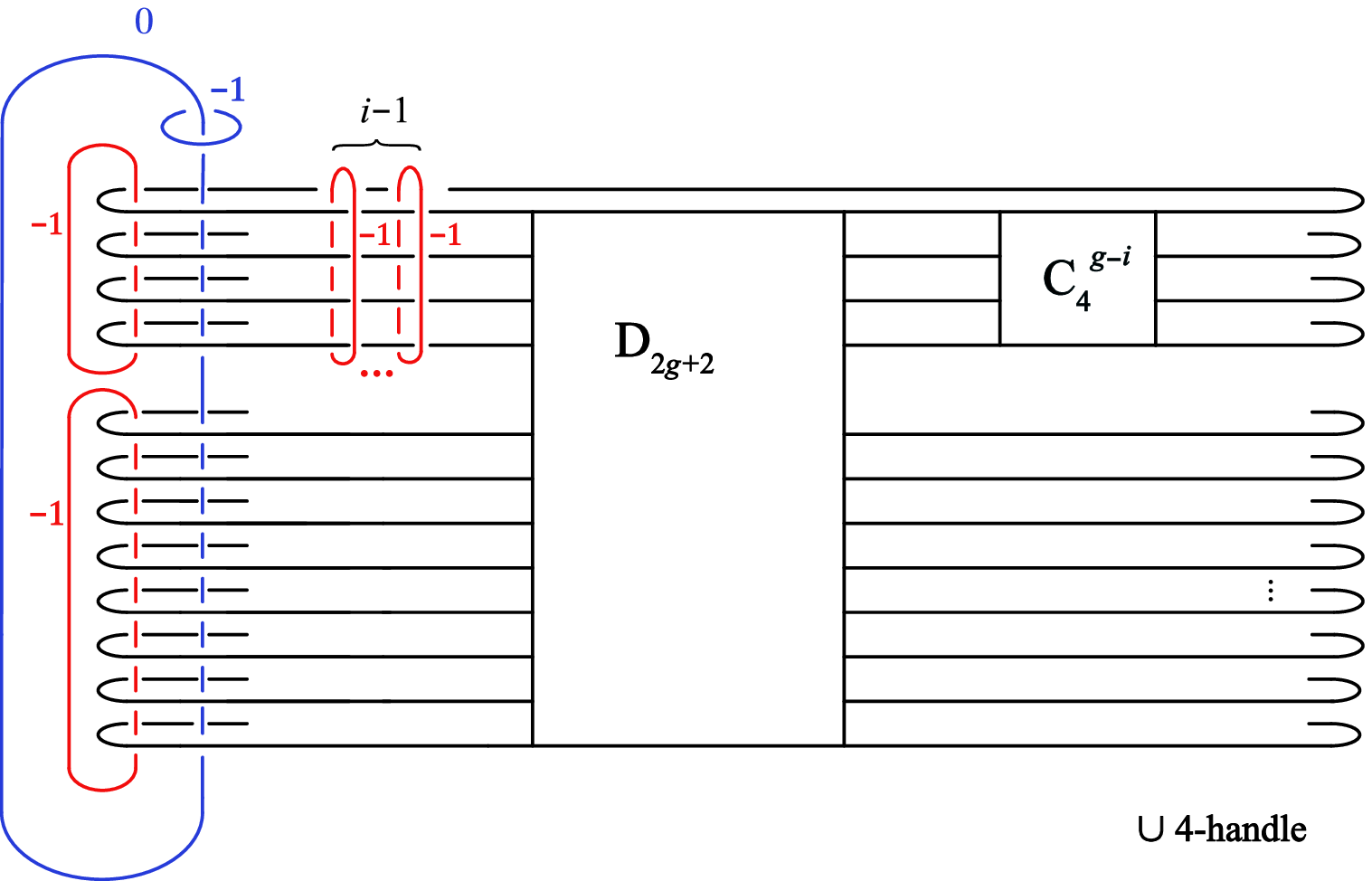}
\caption{}
\label{base_odd_1}
\end{center}
\end{figure}

We next slide the upper left $-1$-framed 2-handle over the lower one, producing Figure~\ref{base_odd_2}.
\begin{figure}[ht]
\begin{center}
\includegraphics[width=4.8in]{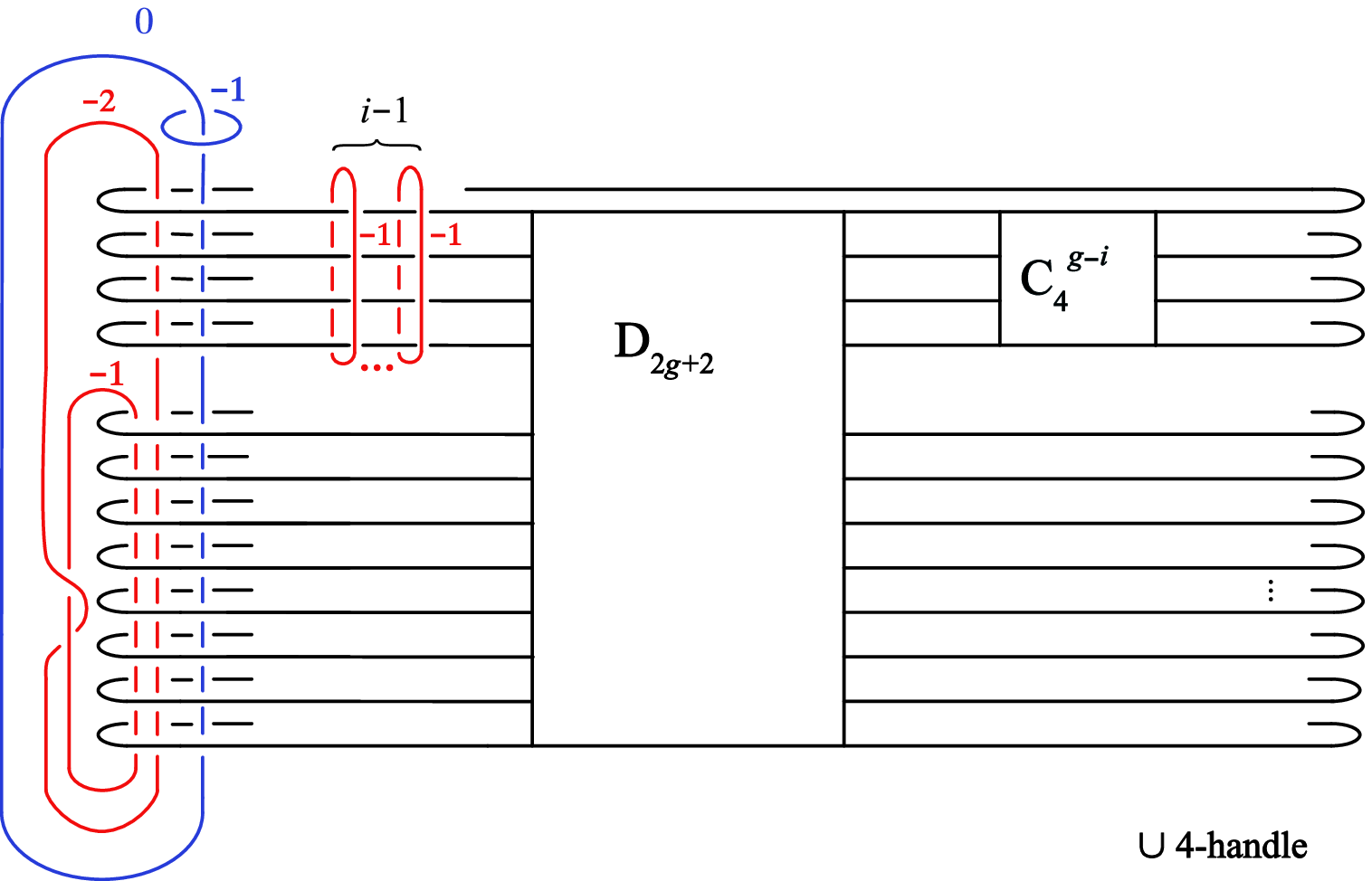}
\caption{}
\label{base_odd_2}
\end{center}
\end{figure}
Next the $-2$-framed 2-handle is slid over the parallel $0$-framed one, giving Figure \ref{base_odd_3}, and then slid over the $-1$-framed 2-handle that links it as a meridian. The result is Figure~\ref{base_odd_4}.

\begin{figure}[ht]
\begin{center}
\includegraphics[width=4.8in]{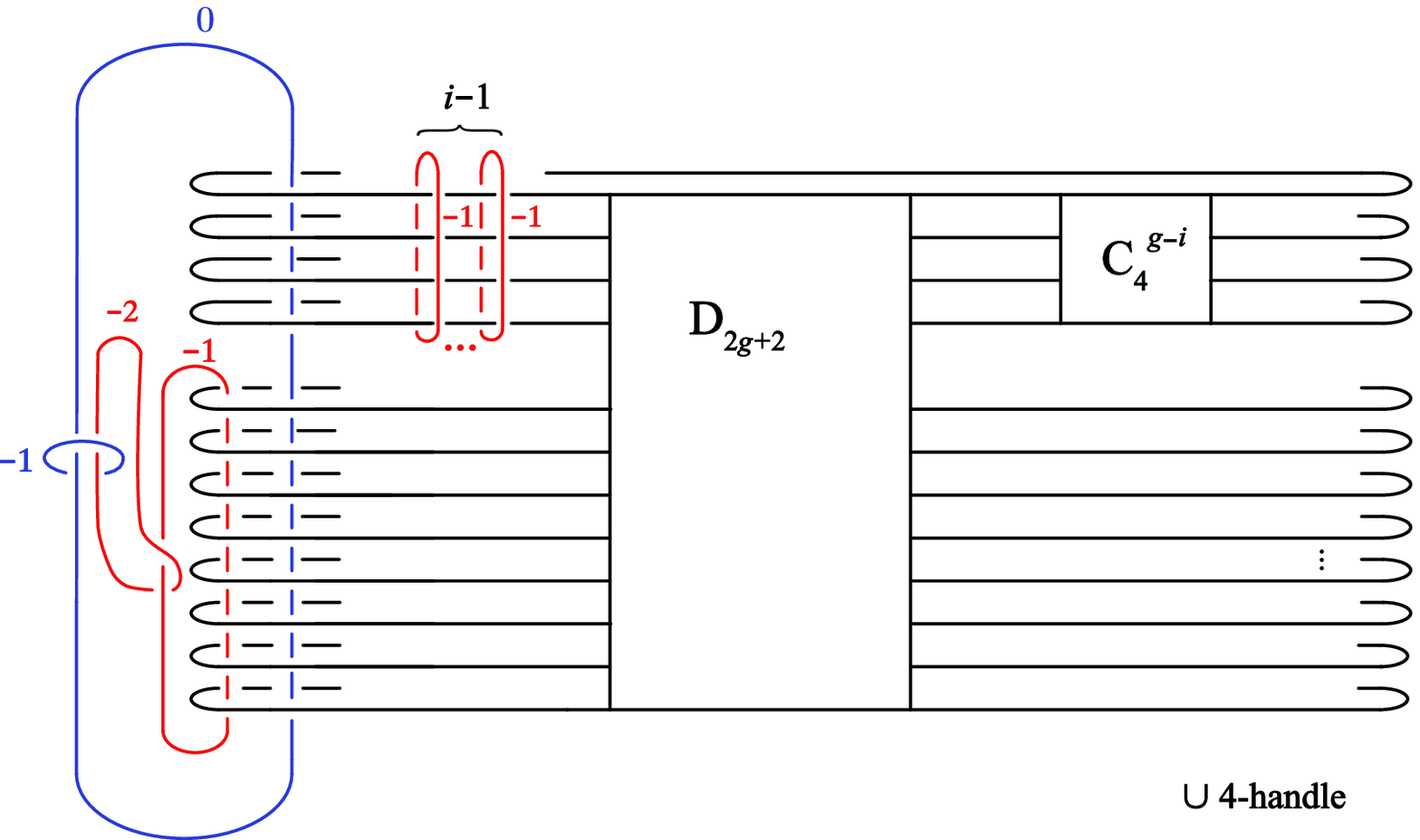}
\caption{}
\label{base_odd_3}
\end{center}
\end{figure}

\begin{figure}[ht]
\begin{center}
\includegraphics[width=4.8in]{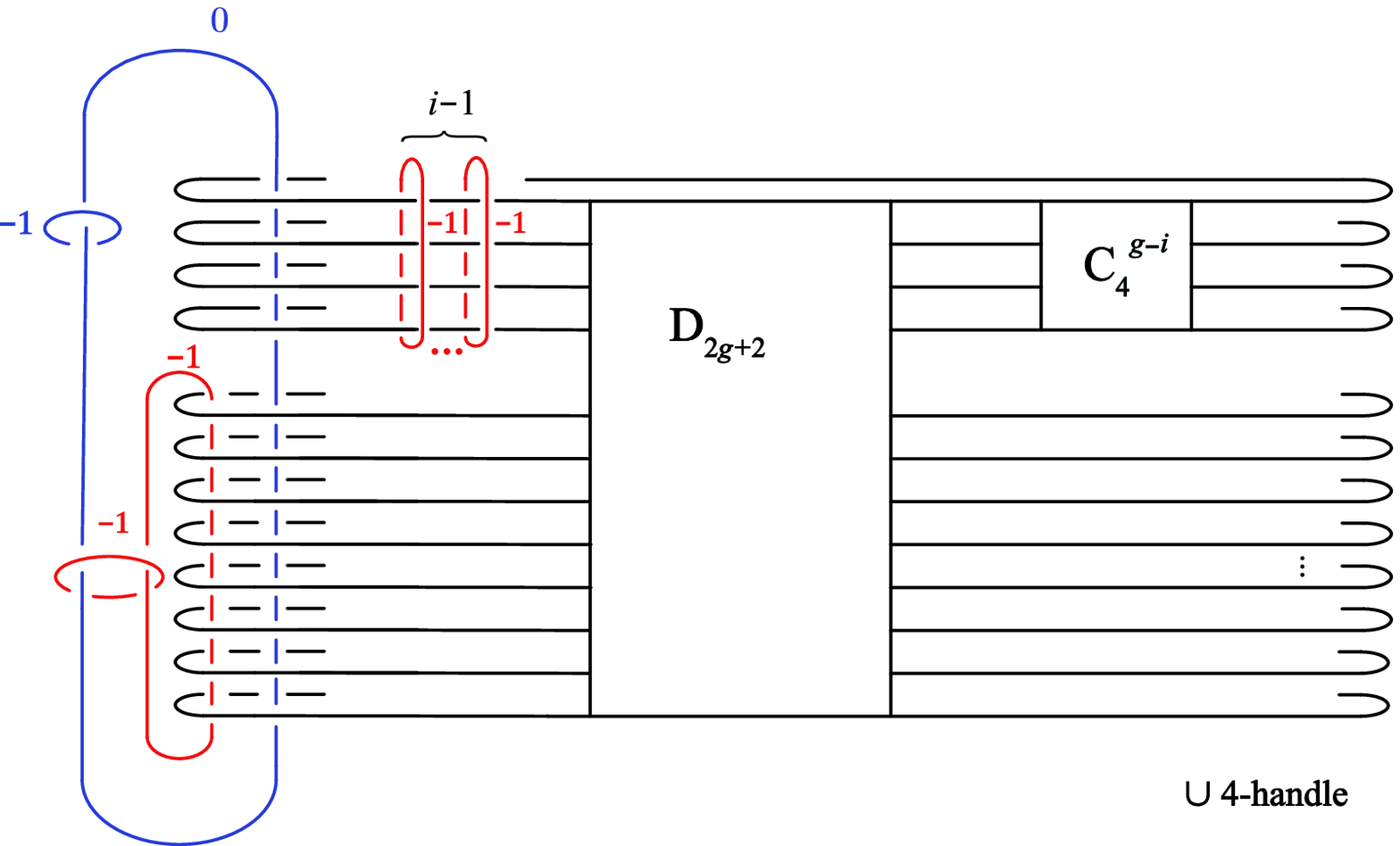}
\caption{}
\label{base_odd_4}
\end{center}
\end{figure}

We repeat this series of slides for each of the remaining $-1$-framed 2-handles at the top of the picture, resulting in Figure~\ref{base_odd_5}. 

\begin{figure}[ht]
\begin{center}
\includegraphics[width=5in]{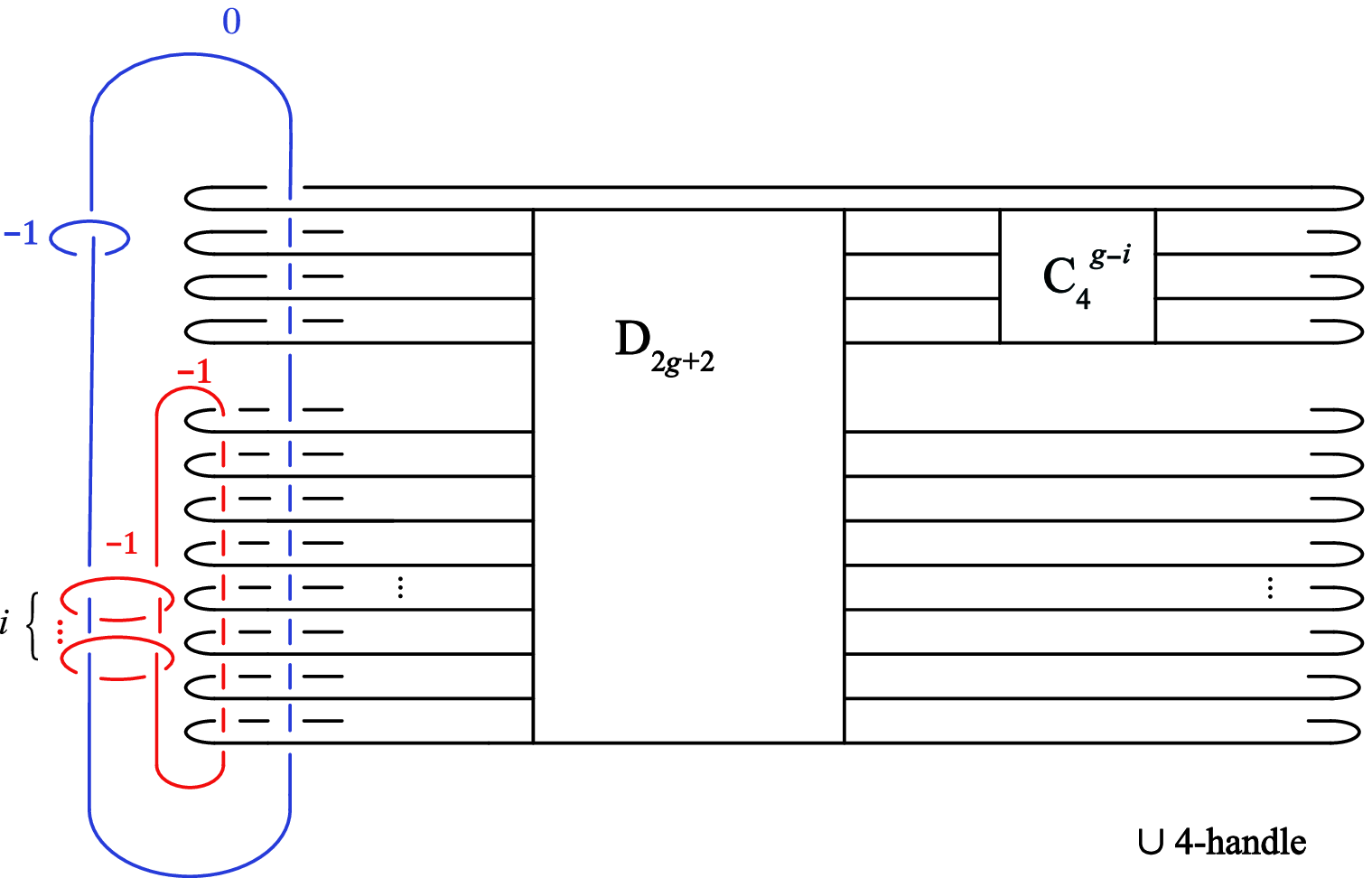}
\caption{}
\label{base_odd_5}
\end{center}
\end{figure}

Next we slide the lower $-1$-framed 2-handle over the blue $0$-framed 2-handle, then slide the result over the (blue) $-1$-framed 2-handle, giving Figure~\ref{base_odd_6}. Finally, we blow down each of the $-1$-framed 2-handles that link the $0$-framed 2-handle, to arrive at Figure~\ref{base_odd_7}.
\begin{figure}[ht]
\begin{center}
\includegraphics[width=4.8in]{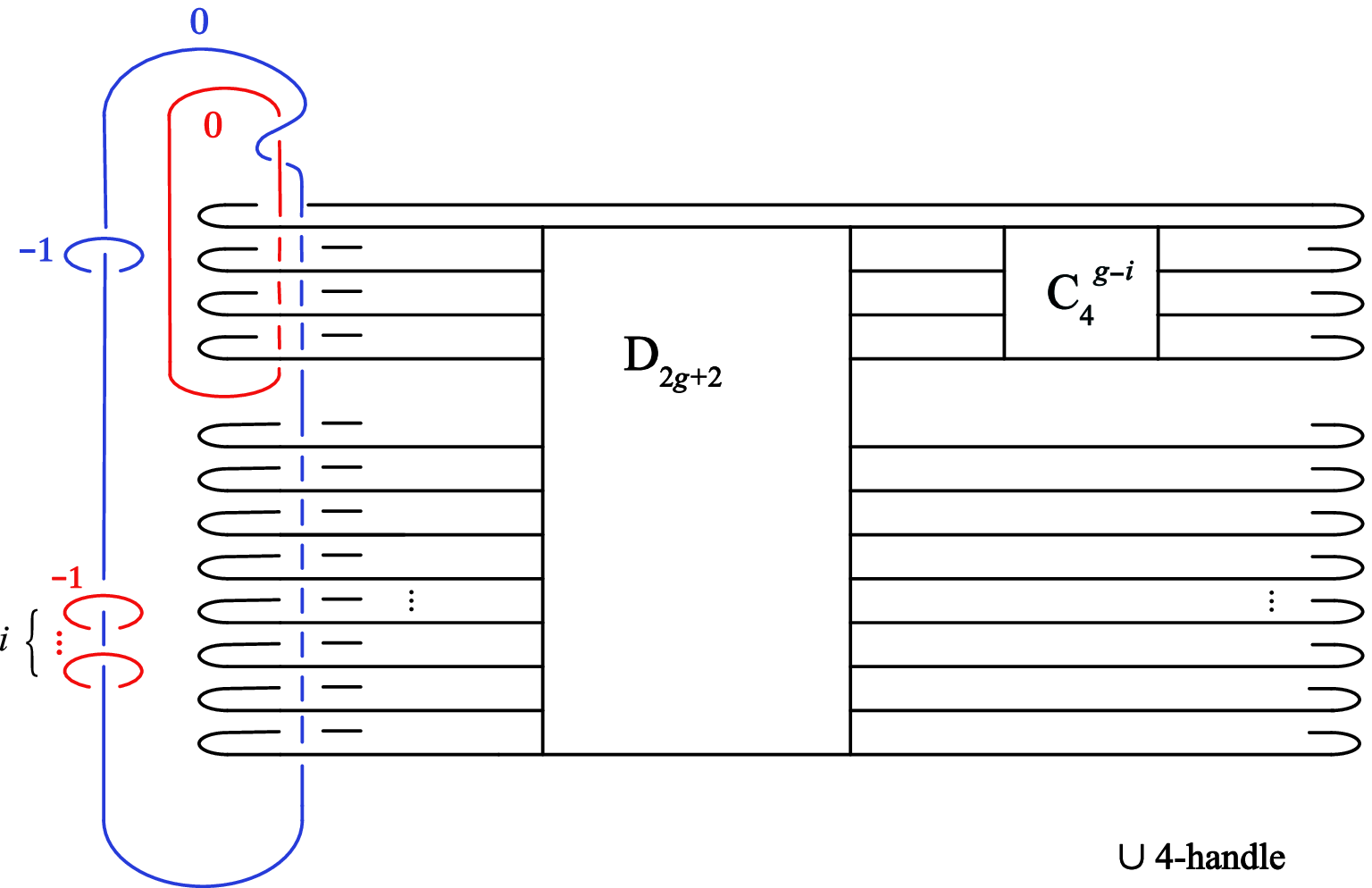}
\caption{}
\label{base_odd_6}
\end{center}
\end{figure}

\begin{figure}[ht]
\begin{center}
\includegraphics[width=5.5in]{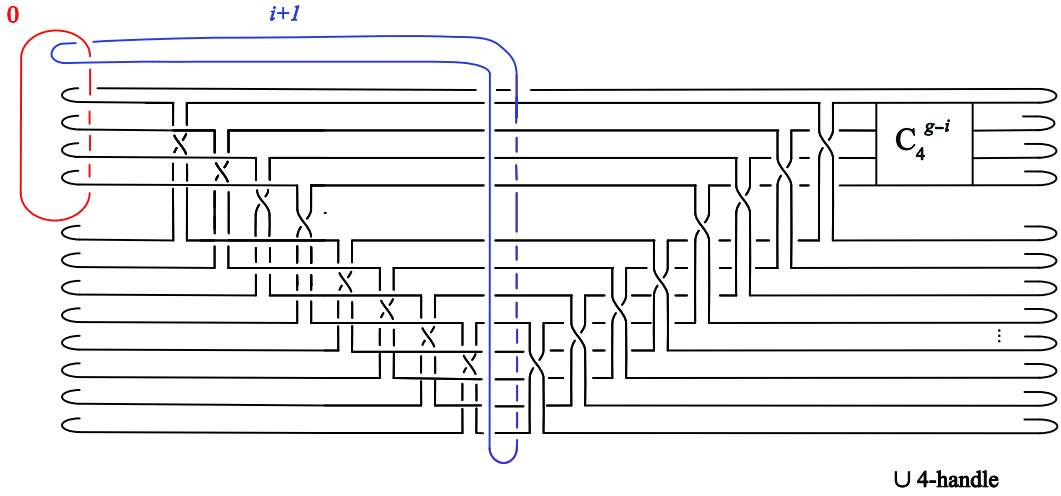}
\caption{The 2-fold branched cover is $X_g'(i)$.}
\label{base_odd_7}
\end{center}
\end{figure}

We pause here for an important observation: in this last step, each of the 2-handles that we are blowing down are attached along meridians to the 0-framed 2-handle. Retracing the diffeomorphism that goes between Figures \ref{base_odd_7} and \ref{base_odd_0}, we see that the spheres given by these handles will each lift to two sections of the Lefschetz fibration on $X_g(i)$, of square $-1$. Because we have blown down $2(i+1)$ sections of the fibration $X_g(i)$ with square $-1$, it follows that the 2-fold branched cover of $\F_{i+1}$ branched over the embedded surface described in Figure~\ref{base_odd_7} is $X_g'(i)$.

We next show that description of $X_g'(i)$ as the branched cover in Figure~\ref{base_odd_7} can be used to show that it is diffeomorphic to $E(g-i)$. This relies on a key lemma.

\subsubsection{A Key Lemma}
To set up the statement, let $F(R,S,T)$ denote any ribbon surface in the 4-manifold $\F_n$ of the form shown in Figure~\ref{lemma_0}. The box can represent any collection of bands, with the condition that any bands located there are attached to the top four horizontal disks, and avoid the disks below.

\begin{figure}[h!]
\begin{center}
\includegraphics[width=5.5in]{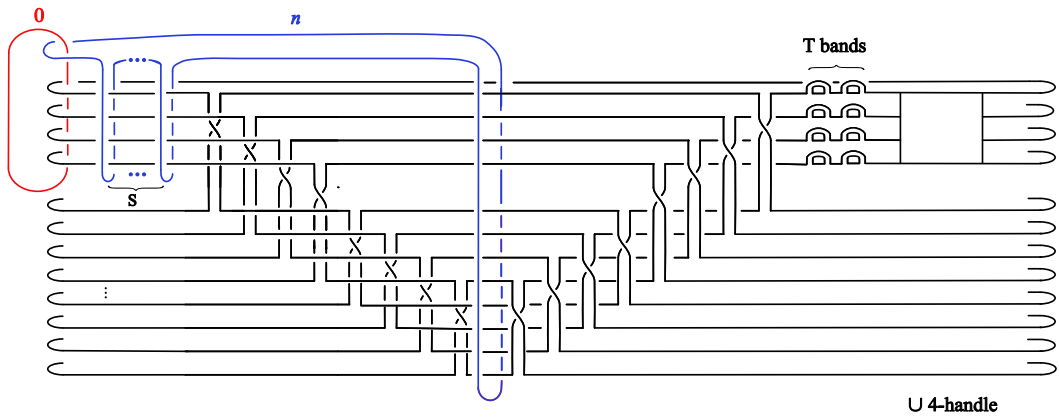}
\caption{The ribbon surface $F(R,S,T)$ has $R$ horizontal disks.}
\label{lemma_0}
\end{center}
\end{figure}
The notation records that: 
\begin{itemize}
\item the ribbon surface has $R$ horizontal disks;
\item the $n$-framed attaching circle links the horizontal disks $S$ times positively in the indicated region; and
\item there are $T$ trivial bands attached to the top four horizontal disks.
\end{itemize}
In applications of Lemma~\ref{recursive}, $T$ will be divisible by four, and the trivial bands will be distributed evenly among the top four horizontal disks.

\begin{Lem} \label{recursive} For $R \geq 8$, the ribbon surface $F(R,S,T)$ is isotopic to the ribbon surface \newline $F(R-4, S+1, T+4)$.
\end{Lem}
 
\begin{proof}
Beginning with $F(R,S,T)$ as shown in Figure~\ref{lemma_0}, we obtain  Figure~\ref{lemma_1} by a 2-handle band dive of the $n$-framed 2-handle. This increases the linking in the upper left of the picture to $S+1$.
\begin{figure}[ht]
\begin{center}
\includegraphics[width=6in]{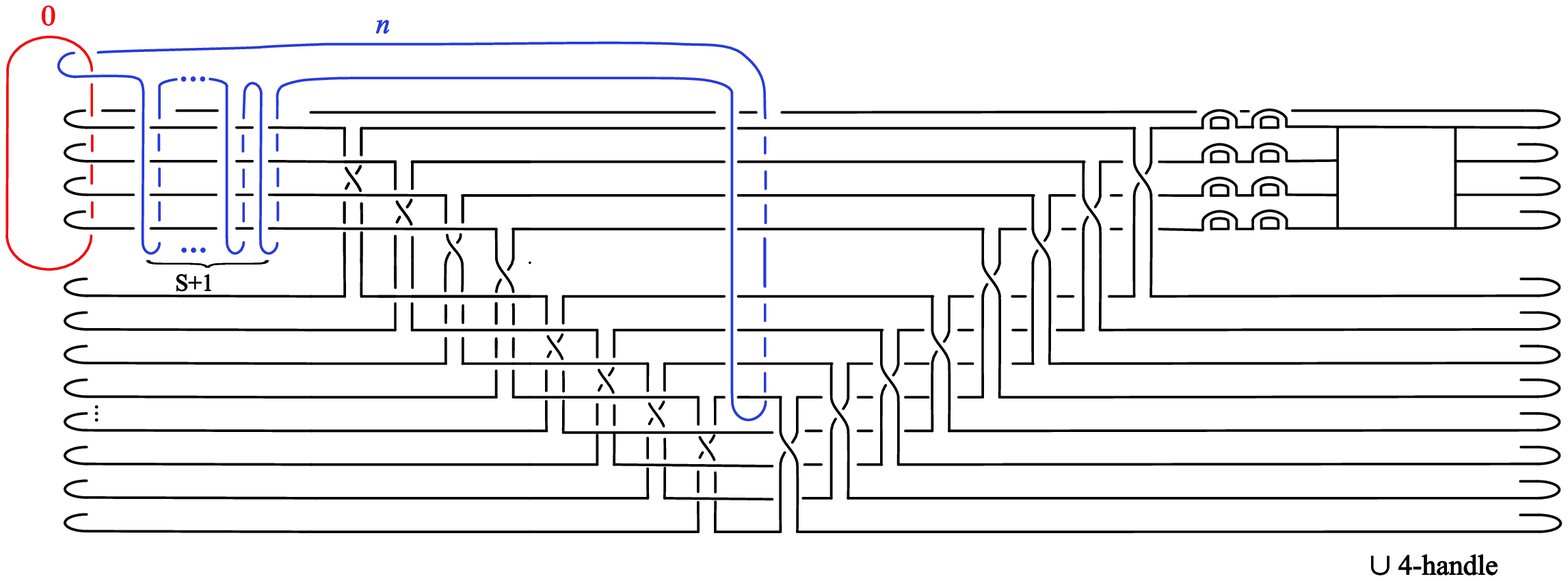}
\caption{}
\label{lemma_1}
\end{center}
\end{figure}
A band slide results in Figure~\ref{lemma_2}, and a band dive of that same band gives Figure~\ref{lemma_3}.
\begin{figure}[ht]
\begin{center}
\includegraphics[width=6in]{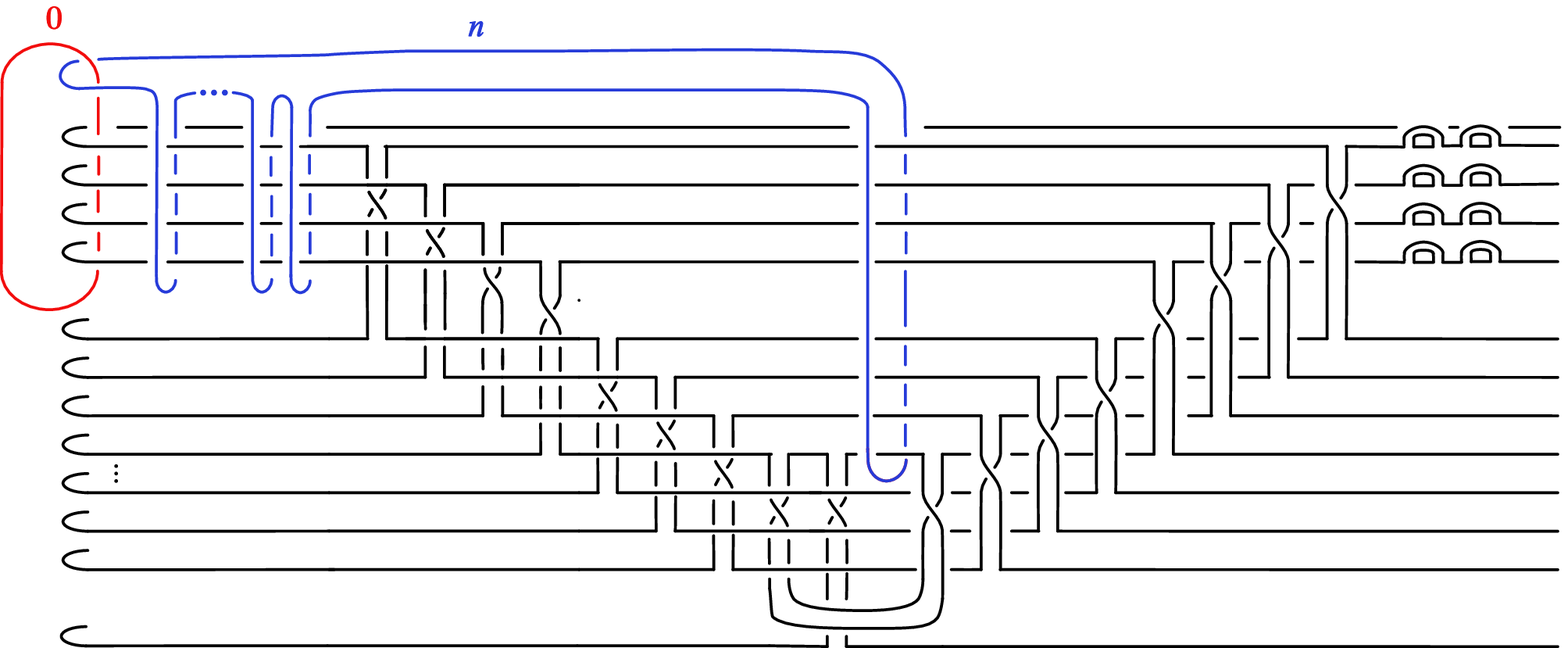}
\caption{}
\label{lemma_2}
\end{center}
\end{figure}

\begin{figure}[ht]
\begin{center}
\includegraphics[width=6in]{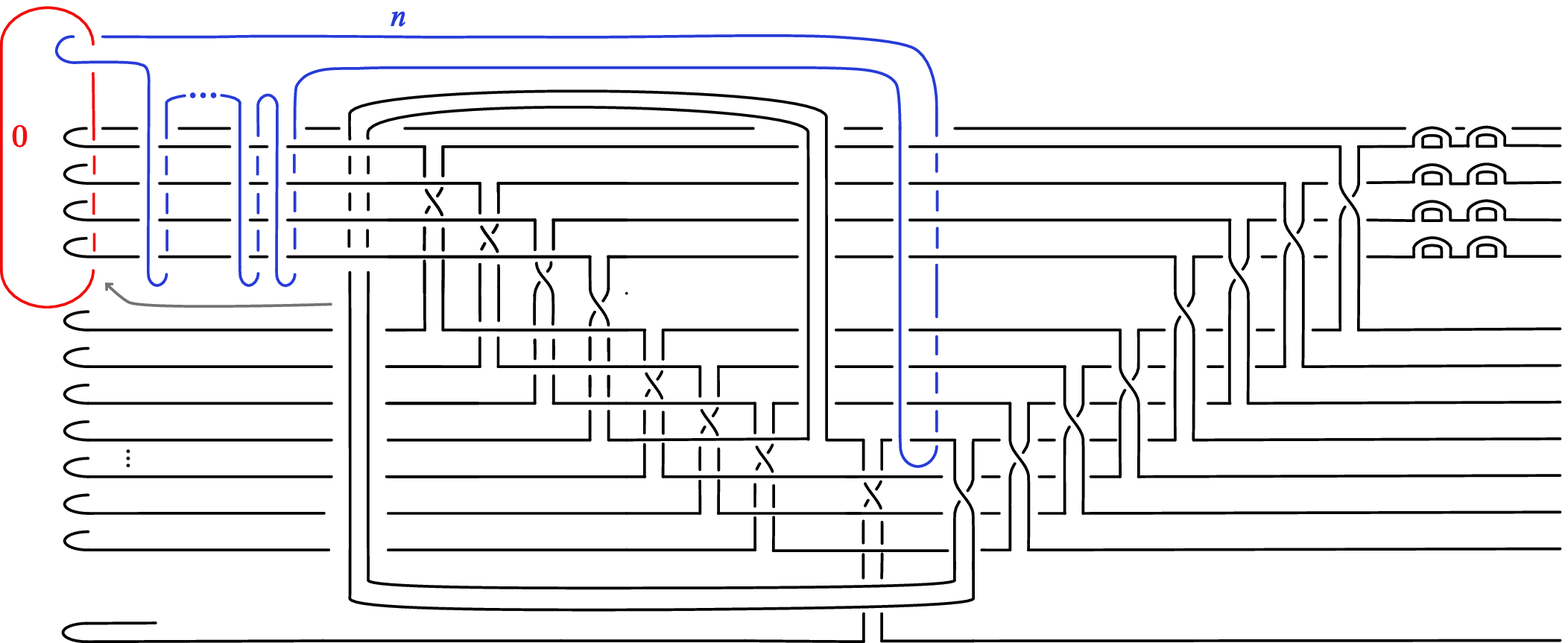}
\caption{}
\label{lemma_3}
\end{center}
\end{figure}
At this point, we may cancel the bottom horizontal disk with the remaining attached band. In addition, we do a 2-handle band slide over the $0$-framed 2-handle, using a band indicated by the grey arrow; the slide disengages the band from the top four horizontal disks, and it can be isotoped to the trivial band shown in Figure~\ref{lemma_4}. 
\begin{figure}[ht]
\begin{center}
\includegraphics[width=6in]{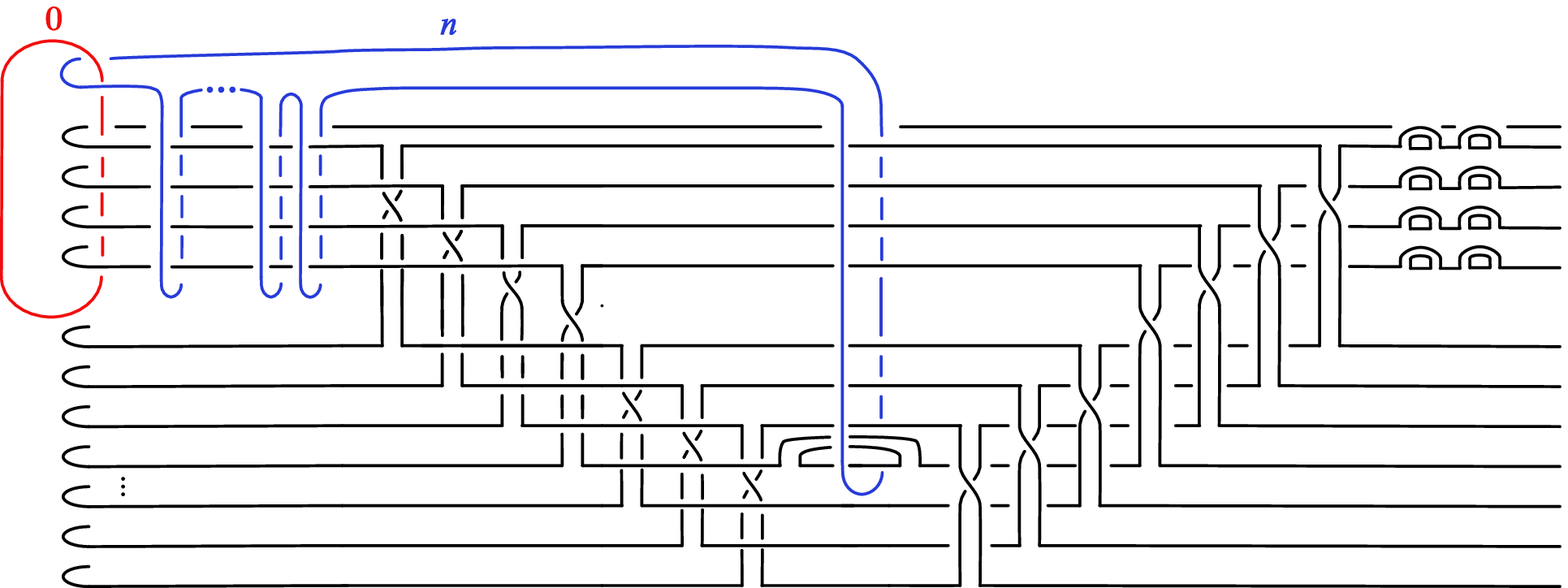}
\caption{}
\label{lemma_4}
\end{center}
\end{figure}

The transition from Figure~\ref{lemma_1} to Figure~\ref{lemma_4} resulted in the cancellation of the bottom horizontal disk, and added a trivial band in the process. We can repeat these steps three times to remove the bottom three horizontal disks, as shown in Figure~\ref{lemma_5}. 
\begin{figure}[ht]
\begin{center}
\includegraphics[width=5in]{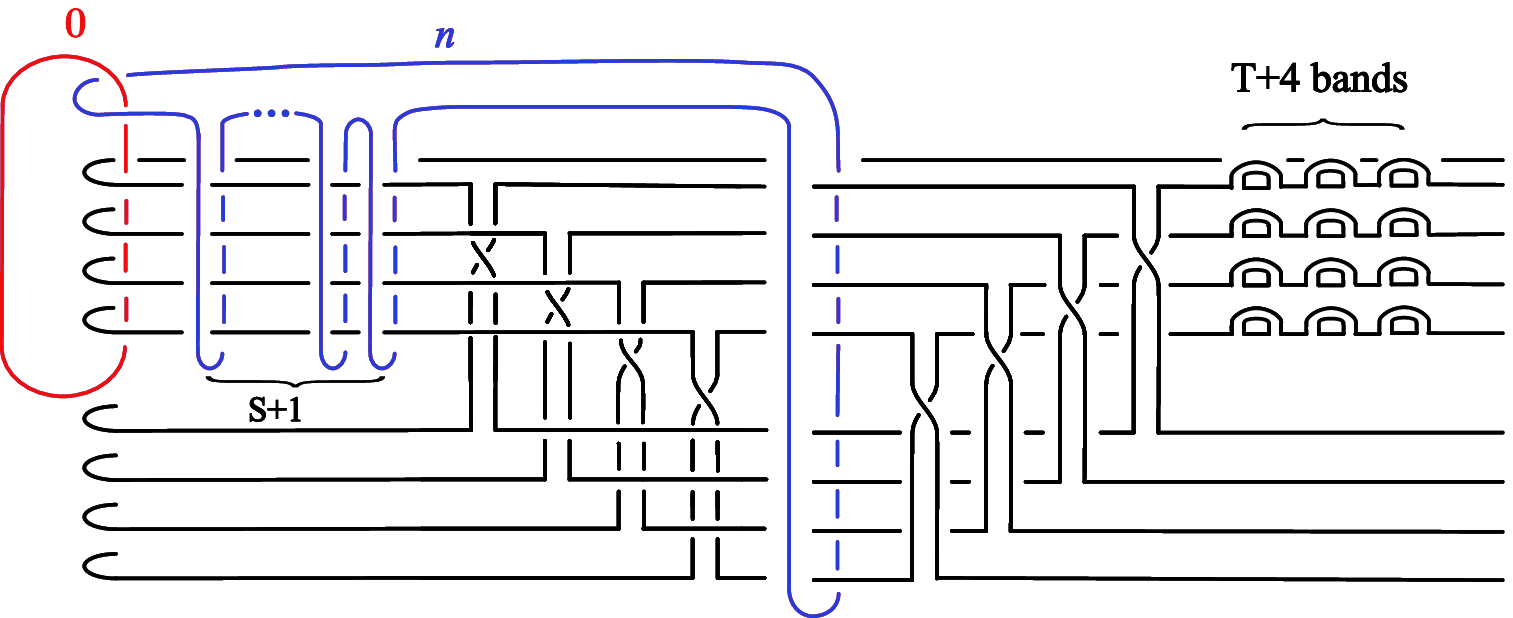}
\caption{}
\label{lemma_5}
\end{center}
\end{figure}
In this Figure, we have moved the trivial bands from their position in Figure~\ref{lemma_4}, by sliding them over the long bands to their right, so that they are now attached to the top four disks. In total we have removed the four bottom horizontal disks, and added four trivial bands; thus the values of $R$ and $T$ change to $R-4$ and $T+4$, respectively. This completes the proof of Lemma~\ref{recursive}.
\end{proof}

\subsubsection{An isotopy of the branch surface}
Let $g=2k+1$. Returning to the proof of Theorem~\ref{main}, Figure~\ref{base_odd_7} shows that $X_g'(i)$ is diffeomorphic to the 2-fold branched cover of $\F_{i+1}$ branched over a surface of the form  $F(2g+2,0,0)=F(4k+4, 0,0)$. Then $k$ iterations of Lemma~\ref{recursive} give that $X_g'(i)$ is diffeomorphic to the 2-fold branched cover branched over a surface of the form $F(4,k,4k)$. Recall that the full surface in Figure~\ref{base_odd_7} includes $2g+2$ unseen disks attached to the boundary of the ribbon surface, with their interiors in the 4-handle. Using $4k=2g-2$ of these disks to cancel the trivial bands, we have that $X'_g(i)$ is diffeomorphic to the cover of the manifold in Figure~\ref{base_odd_9}. (Note that 4 disks remain in the 4-handle.) 
\begin{figure}[ht]
\centering
\begin{minipage}{.45\textwidth}
  \centering
  \includegraphics[width=\textwidth]{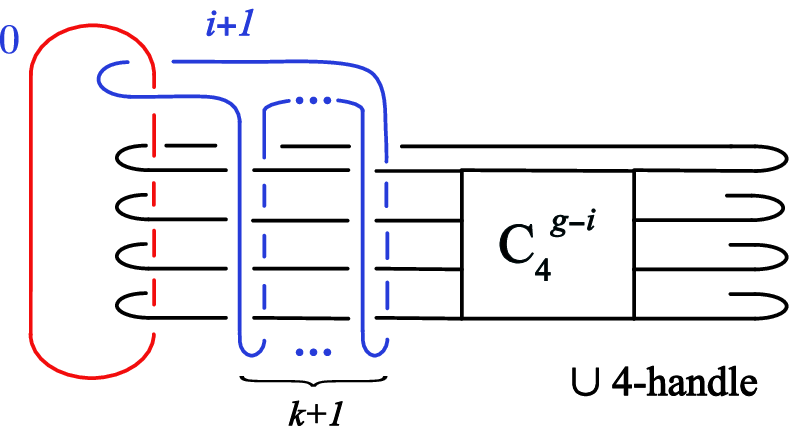}
  \caption{}
  \label{base_odd_9}
\end{minipage} \hfill
\begin{minipage}{.45\textwidth}
  \centering
  \includegraphics[width=.8\textwidth]{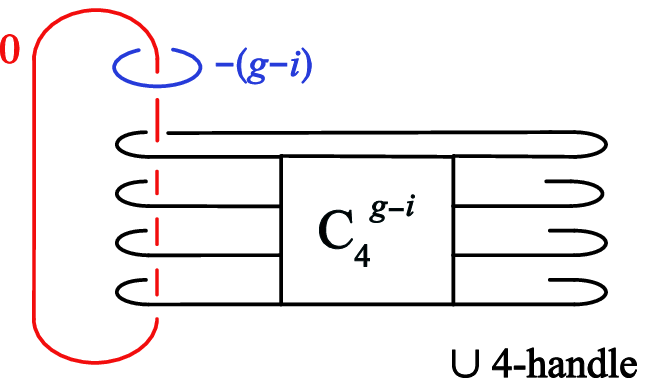}
  \caption{}
  \label{base_odd_10}
\end{minipage}
\end{figure}
We arrive at Figure~\ref{base_odd_10} by sliding the $(i+1)$-framed 2-handle over the 0-framed one $k+1$ times.
The new framing is $(i+1)-2(k+1)=i-2k-1=-(g-i)$, as shown. 
 
The proof for odd $g$ is completed by recognizing that the branched cover of $\F_{g-i}$ over the surface in Figure~\ref{base_odd_10} is $E(g-i)$. This is immediate from the discussion in Section~\ref{sec:Lefschetz}. The lift of the branched cover of the 0-handle union the 0-framed 2-handle branched over the ribbon surface is a genus 1 Lefschetz fibration over $D^2$ with monodromy $(t_{c_1} t_{c_2} t_{c_3})^{4(g-i)}$. The braid $(d_{\pi(c_1)} d_{\pi(c_2)} d_{\pi(c_3)})^{4(g-i)}$ is equal to $g-i$ full right-handed Dehn twists about a circle enclosing all branch points. This isotopy of this to the identity fixes a reference point in $\Sigma_{0,4} \setminus B_4^2$ while rotating a framed neighborhood $g-i$ times in a left-handed direction. Thus adding a 2-handle with the indicated location and framing shows that the branched cover of $\F_{g-i}$ over the rest of the surface extends to a total space which is a genus 1 Lefschetz fibration over $S^2$, whose monodromy matches a well known factorization of $E(g-i)$. 

\subsection{The proof for even $g$}

The proof for even $g$ is essentially the same as for odd $g$. However, because $2g+2$ is no longer divisible by four, we must include two additional iterations of the basic moves used in the proof of Lemma~\ref{recursive}. Also, because the different form of the monodromy of $X_g'(i)$ makes for a different ribbon branch surface, the final step of recognizing the total space of the cover as an elliptic surface is somewhat different.

 As a starting point for even $g$, we begin with the Lefschetz fibration on $Z_g$, which from  (\ref{BHMchain_Heq_p}) has a mondoromy factorization given by the relation
\begin{equation} \label{BHMchain_Heq_f_even}
D_g E_g (t_{c_1} t_{c_2} t_{c_3})^{4i} (t_{c_5} t_{c_6} \cdots t_{c_{2g+1}})^{2g-2} (t_{c_1} t_{c_2} t_{c_3})^{4(g-i-1)} (t_{c_1} t_{c_2} t_{c_3})^{2} (t_{c_3} t_{c_2} t_{c_1})^{2} = 1.
\end{equation}
As before, this hyperelliptic Lefschetz fibration can be described as the 2-fold cover of $\F_1$ branched over the surface described in Figures~\ref{base_even_chains} and \ref{Ebraid}. 
\begin{figure}[ht]
%\begin{center}
\includegraphics[width=5in]{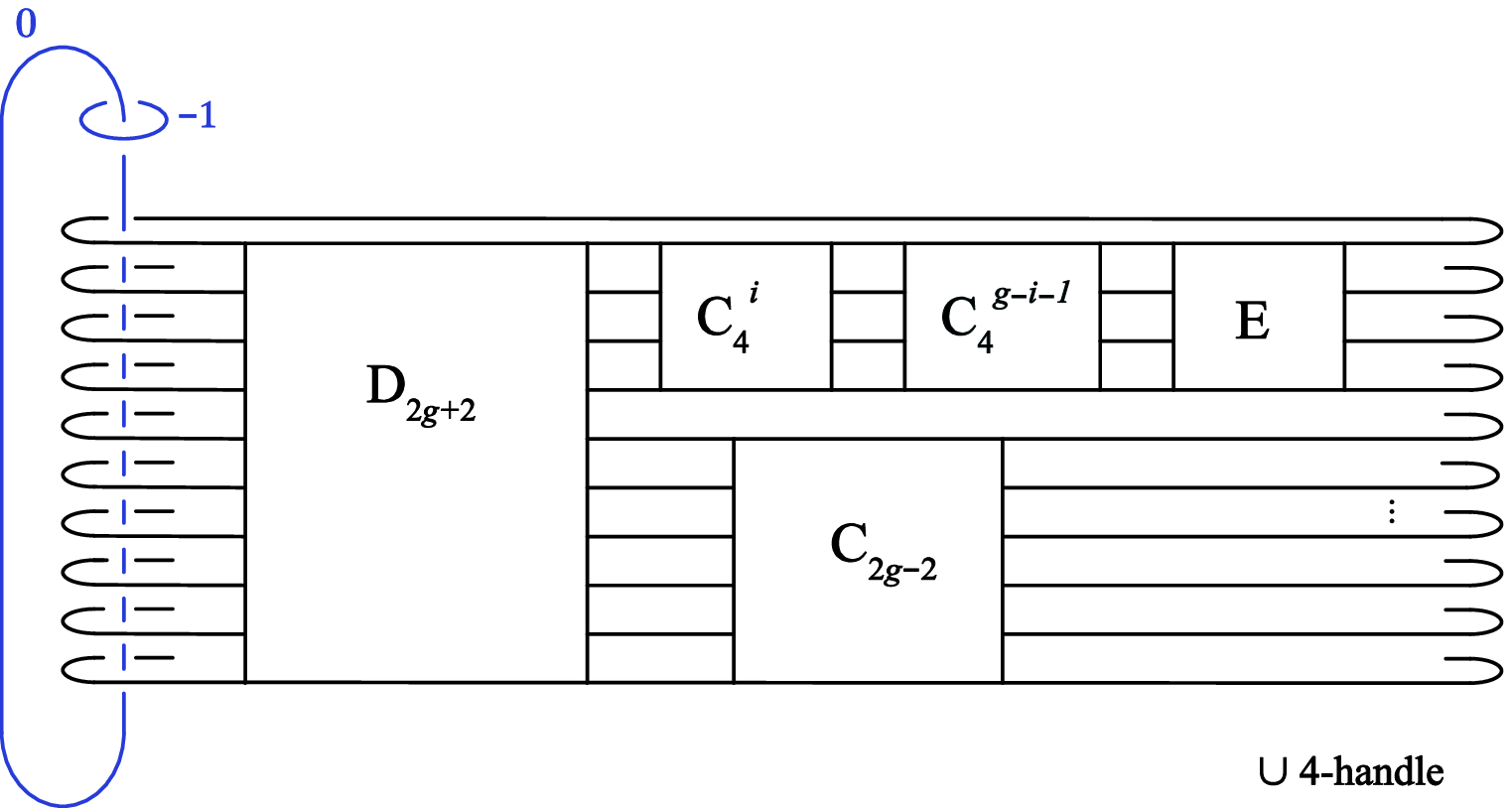}
\caption{}
\label{base_even_chains}
%\end{center}
\end{figure}
\begin{figure}[ht]
%\begin{center}
\includegraphics{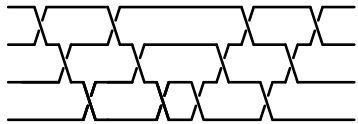}
\caption{The ribbon surface $E$.}
\label{Ebraid}
%\end{center}
\end{figure}
Once again, it can be checked that the projection of (\ref{BHMchain_Heq_f_even}) to a homeomorphism of $\Sigma_{0,2g+2}$ equals a right-handed Dehn twist about a circle that encloses all marked branch points. The unseen part of the branch surface is $2g+2$ disks attached to the boundary of the ribbon surface, with interiors in the 4-handle, exactly as in the odd $g$ case. Thus Figure~\ref{base_even_chains} depicts a banded unlink diagram, as before.

Performing unchaining monodromy substitutions gives that $X_g(i)$ is the 2-fold cover of $\F_1 \# (i+1) \cpb$, branched over the surface in Figure~\ref{base_even_0}. 
\begin{figure}[ht]
\begin{center}
\includegraphics[width=5in]{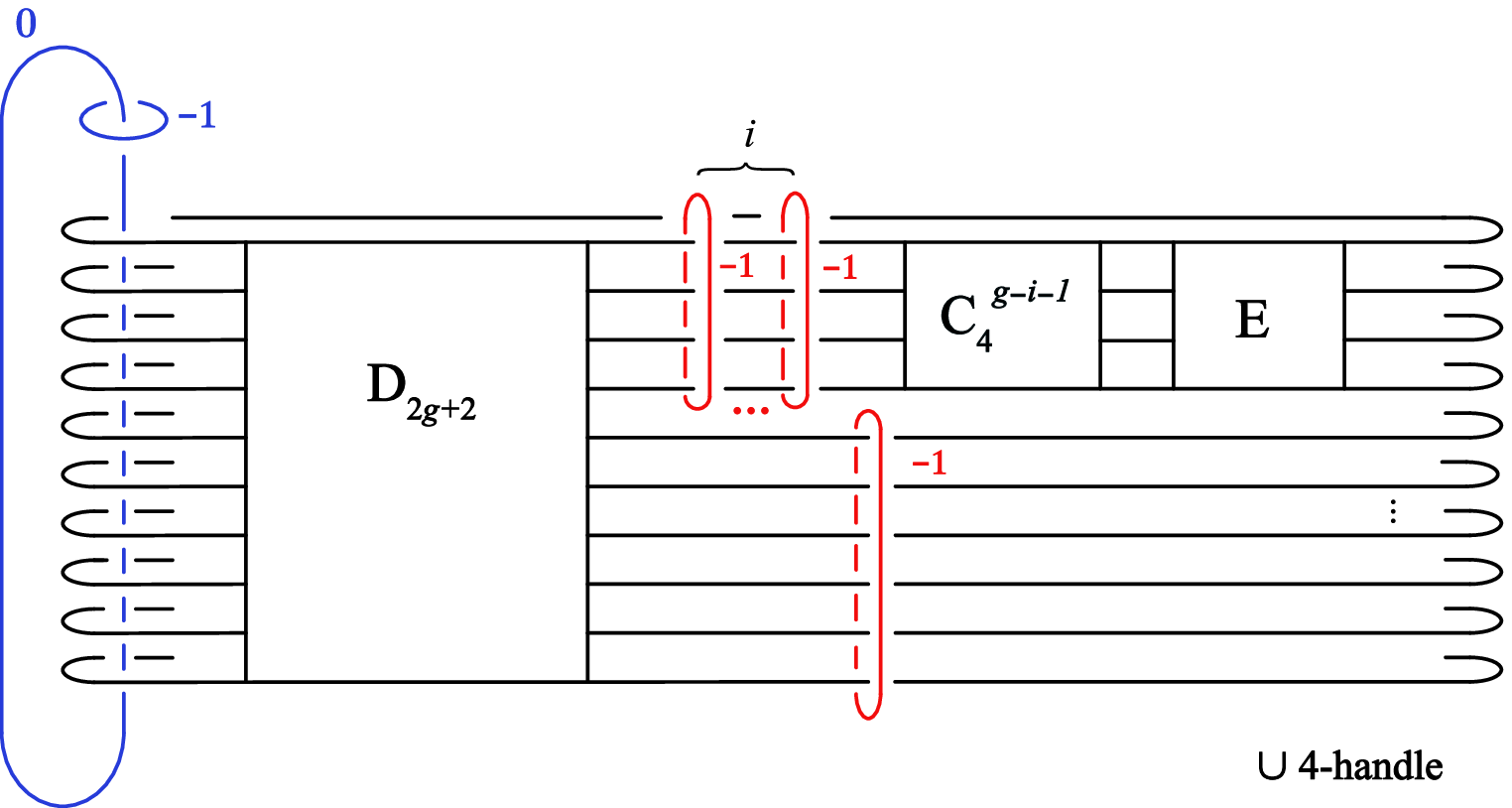}
\caption{}
\label{base_even_0}
\end{center}
\end{figure}
Mimicking the 2-handle slides from the odd case yields $2(i+1)$ sections which are blown down to give $X_g'(i)$ as the 2-fold cover of $\F_{i+1}$ branched over the surface in Figure~\ref{base_even_7}.
\begin{figure}[ht]
\begin{center}
\includegraphics[width=6in]{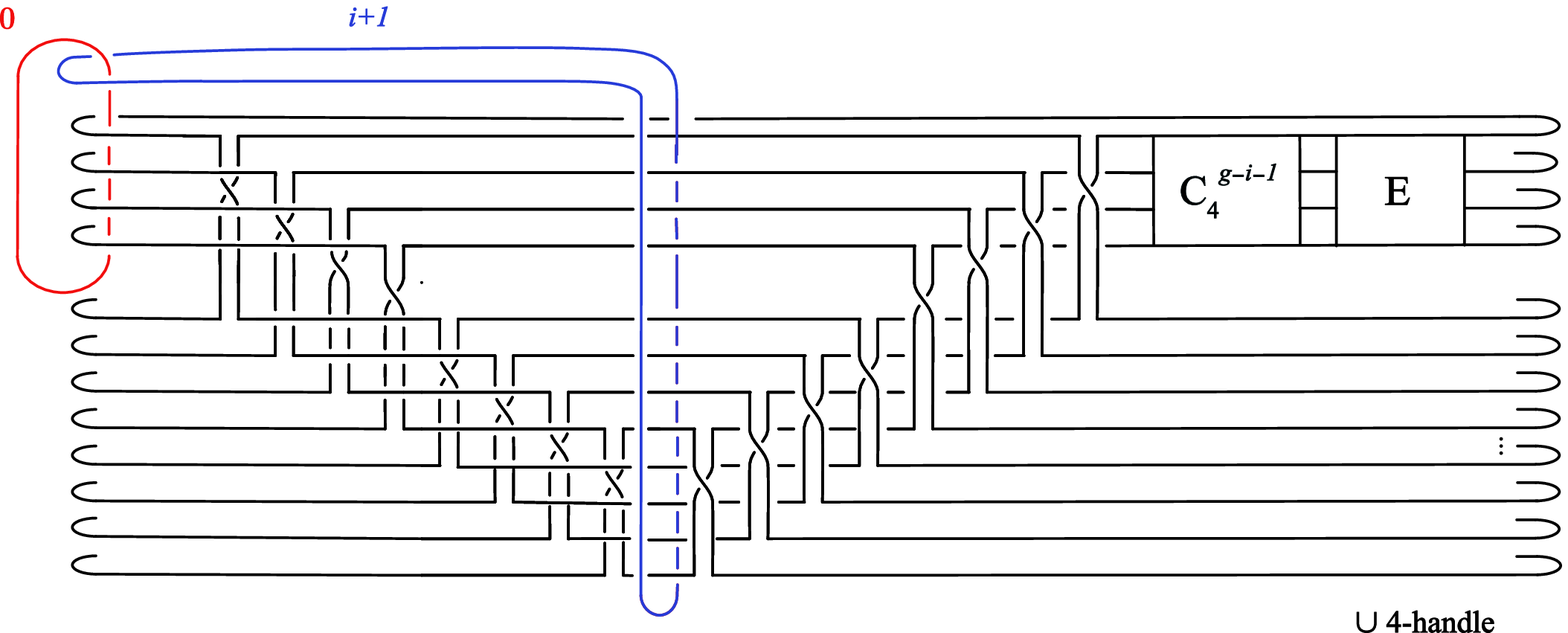}
\caption{}
\label{base_even_7}
\end{center}
\end{figure}

Let $g=2k+2$. The ribbon surface in Figure~\ref{base_even_7} is of the form $F(2g+2,0,0)=F(4k+6,0,0)$. Then $k$ iterations of Lemma~\ref{recursive} give that $X_g'(i)$ is diffeomorphic to the cover of $\F_{i+1}$ branched over $F(6,k,4k)$, shown in Figure~\ref{base_even_8}. 
\begin{figure}[ht]
\begin{center}
\includegraphics[width=4in]{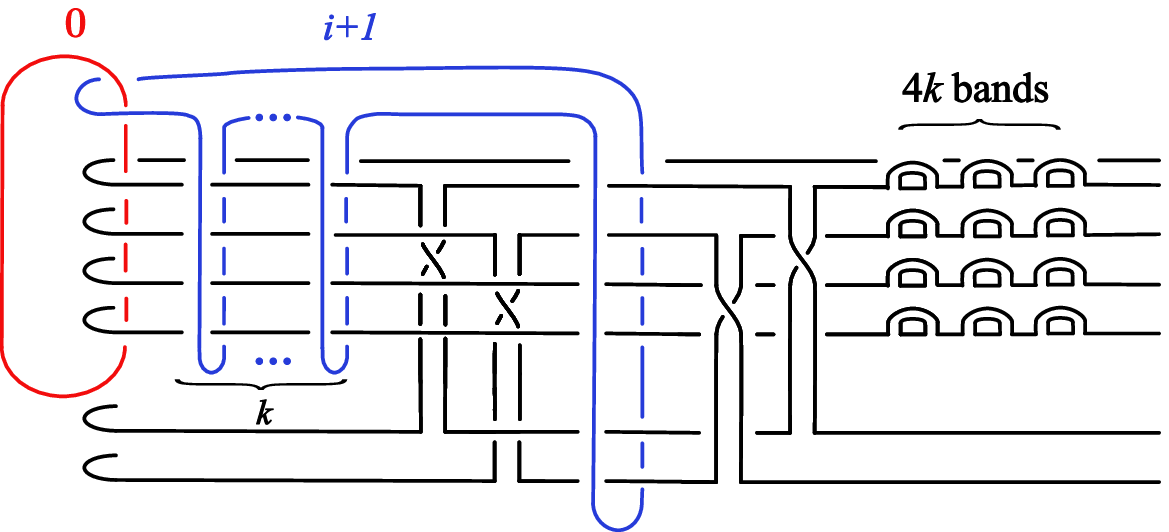}
\caption{}
\label{base_even_8}
\end{center}
\end{figure}

We next cancel the bottom two horizontal disks as follows. A 2-handle band dive gives Figure~\ref{base_even_9}.
\begin{figure}[ht]
\begin{center}
\includegraphics[width=4in]{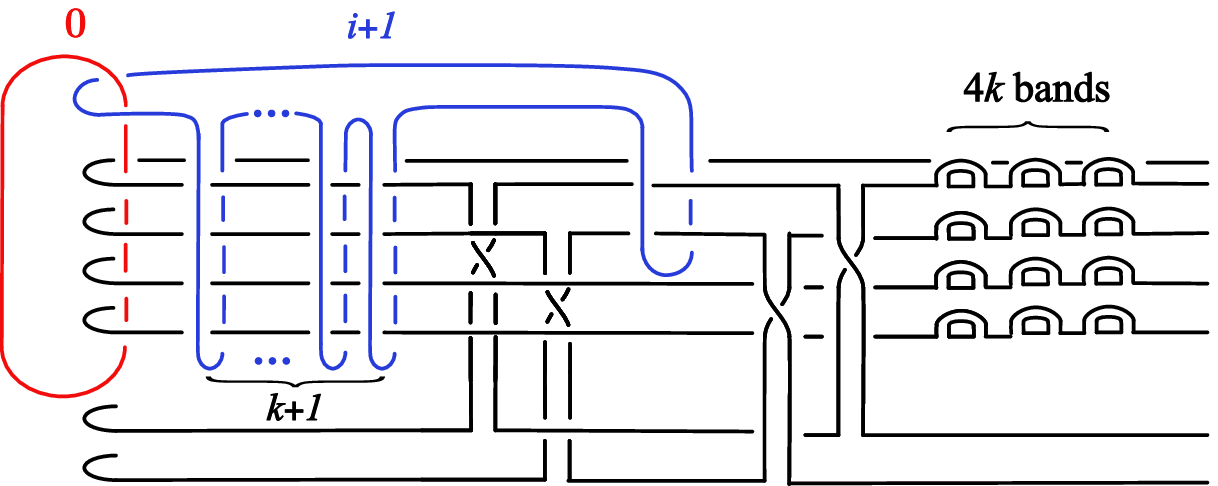}
\caption{}
\label{base_even_9}
\end{center}
\end{figure}
We can then twice more use the sequence of moves in the proof of Lemma~\ref{recursive}: a band slide, followed by a band dive, followed by a 2-handle band slide. (See the transition from Figure~\ref{lemma_1} to Figure~\ref{lemma_4}.) This adds two more trivial bands to the picture, but we cancel all $4k+2=2g-2$ of them using disks from the 4-handle. This results in Figure~\ref{base_even_10}. 
\begin{figure}[ht]
\centering
\begin{minipage}{.45\textwidth}
  \centering
  \includegraphics[width=\textwidth]{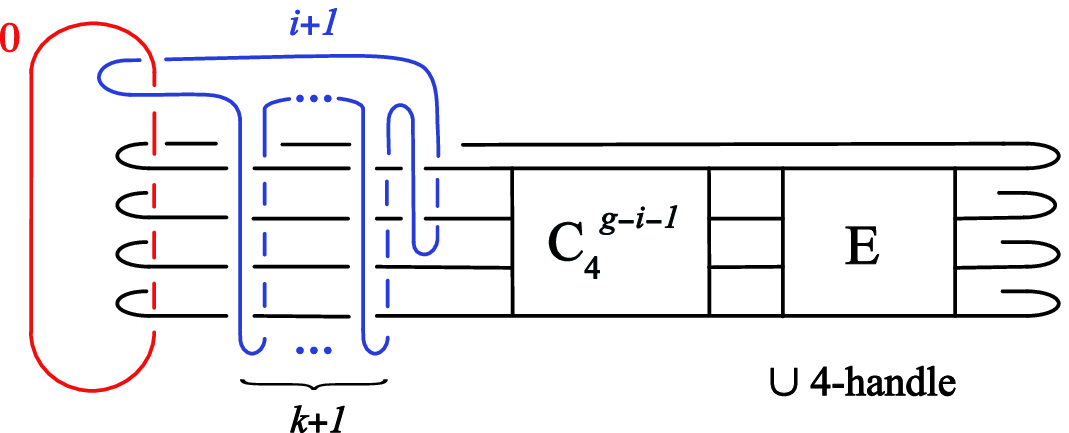}
  \caption{}
  \label{base_even_10}
\end{minipage} \hfill
\begin{minipage}{.45\textwidth}
  \centering
  \includegraphics[width=.85\textwidth]{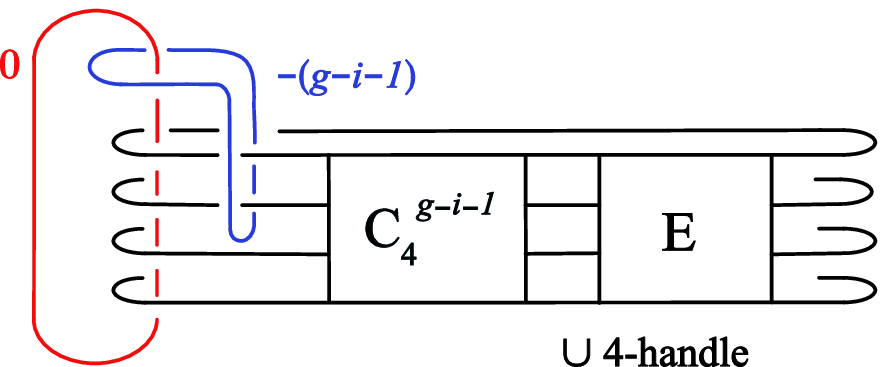}
  \caption{}
  \label{base_even_11}
\end{minipage}
\end{figure}

We next slide the $(i+1)$-framed 2-handle $k+1$ times over the 0-framed handle. The new framing is $(i+1)-2(k+1)=i-2k-1=-(g-i-1)$. This is Figure~\ref{base_even_11}. 

It remains to see that the branched cover described by Figure~\ref{base_even_11} is $E(g-i)$. The lift of the branched cover of the 0-handle union the 0-framed 2-handle branched over the ribbon surface is a genus 1 Lefschetz fibration over $D^2$ with monodromy $$(t_{c_1} t_{c_2} t_{c_3})^{4(g-i-1)} (t_{c_1} t_{c_2} t_{c_3})^2 (t_{c_3} t_{c_2} t_{c_1})^2.$$ The location and framing of the other attaching circle is explained by tracking a framed neighborhood of a reference point $\ast\in\Sigma_{0,2g+2}\setminus B_{2g+2}^2$ through an isotopy from the braid $$(d_{\pi(c_1)} d_{\pi(c_2)} d_{\pi(c_3)})^{4(g-i-1)} (d_{\pi(c_1)} d_{\pi(c_2)} d_{\pi(c_3)})^2 (d_{\pi(c_3)} d_{\pi(c_2)} d_{\pi(c_1)})^2$$ to the identity. This isotopy first undoes $g-i-1$ right handed Dehn twists, which fixes $\ast$ while rotating its neighborhood $g-i-1$ times oppositely; followed by an isotopy that pushes $\ast$ around a circle passing through the middle two marked points without twisting its neighborhood. Thus the branched cover of $\F_{g-i-1}$ extended over the rest of the surface gives a total space which is a genus 1 Lefschetz fibration over $S^2$. Finally, we note that the monodromy factorization of this fibration is easily seen to be equivalent to other well-known factorizations for elliptic fibrations on $E(g-i)$.

\end{document}